%% file: main_hdiv.tex
\documentclass{article}
\usepackage[utf8]{inputenc}
\usepackage[T1]{fontenc}
\usepackage{amsmath}
\usepackage{amsfonts}
\usepackage{hyperref}
\hypersetup{
    colorlinks=true,
    linkcolor=magenta,
    filecolor=magenta,      
    urlcolor=cyan,
    pdfpagemode=FullScreen,
    }
\usepackage{amssymb}
\usepackage{mathtools}
\usepackage[version=4]{mhchem}
\usepackage{tabularx}
\usepackage[a4paper, total={6.5in, 9in}]{geometry}
\usepackage[most]{tcolorbox}
\usepackage{xcolor}
\usepackage[numbers]{natbib}
\usepackage{stmaryrd}
\usepackage{bbold}
\usepackage{multirow} 
\usepackage{enumitem}
\usepackage{xcolor}
\usepackage{hyperref}
\usepackage{tikz-cd}
\usepackage{epstopdf}
\usepackage{bm}
\usepackage{pgfplots}
\usepackage{subcaption}
\usepackage{graphicx}
\usepackage[export]{adjustbox}
\usepackage{amsthm}
\usepackage{chngcntr}
\usepackage{stmaryrd}
\usepackage{apptools}
\usepackage{esint}
\usepackage{tikz}
\usetikzlibrary{patterns,positioning,cd,calc}
\usepackage{authblk}

\newtheorem{theorem}{Theorem}[section]
\newtheorem{lemma}[theorem]{Lemma}

\newtheorem{corollary}[theorem]{Corollary}

\newtheorem{remark}{Remark}

\newtheorem*{assumption}{Assumption}

\numberwithin{equation}{section}

\def\div{\mathrm{div\,}}
\def\curl{\boldsymbol{\mathrm{curl}}\,}
\def\grad{\boldsymbol{\mathrm{grad}}\,}
\def\vecDelta{\boldsymbol{\Delta}}
\def\divh{\mathrm{div}_h\,}
\def\curlh{\boldsymbol{\mathrm{curl}}_h\,}
\def\gradh{\boldsymbol{\mathrm{grad}}_h\,}
\def\divhz{\mathrm{div}^\circ_h\,}
\def\curlhz{\boldsymbol{\mathrm{curl}}^\circ_h\,}
\def\gradhz{\boldsymbol{\mathrm{grad}}^\circ_h\,}

\def\Hdiv{{\boldsymbol{H}(\mathrm{div})}}
\def\Hzdiv{{\boldsymbol{H}_0(\mathrm{div})}}
\def\Hdivz{{\boldsymbol{H}(\mathrm{div}0)}}
\def\Hzdivz{{\boldsymbol{H}_0(\mathrm{div}0)}}
\def\Hcurl{{\boldsymbol{H}(\boldsymbol{\mathrm{curl}})}}
\def\Hzcurl{\boldsymbol{H}_0(\boldsymbol{\mathrm{curl}})}
\def\Hcurlz{\boldsymbol{H}(\boldsymbol{\mathrm{curl}}\boldsymbol{0})}
\def\Hzcurlz{\boldsymbol{H}_0(\boldsymbol{\mathrm{curl}}\boldsymbol{0})}
\def\Ltwo{L^2}
\def\Ltwoz{L^2_0}

\def\ub{{\boldsymbol{u}}}
\def\vb{{\boldsymbol{v}}}
\def\mub{{\boldsymbol{\mu}}}
\def\nb{{\boldsymbol{n}}}

\def\R{{\mathbb{R}}}
\def\fb{\boldsymbol{f}}
\def\taub{\boldsymbol{\tau}}
\def\psib{\boldsymbol{\psi}}
\def\Hb{\boldsymbol{H}}
\def\Lb{\boldsymbol{L}}
\def\a{\mathsf{a}}
\def\b{\mathsf{b}}
\def\Nedelec{N\'{e}d\'{e}lec }
\def\xb{\boldsymbol{x}}
\def\dist{\mathrm{dist}}
\def\xib{\boldsymbol{\xi}}
\def\Wb{\boldsymbol{W}}
\def\supp{\mathrm{supp}}

\def\wb{\boldsymbol{w}}
\def\lambdab{\boldsymbol{\lambda}}
\def\Poincare{Poincar\'{e} }
\def\BNB{Banach-Ne\v cas-Babu\v ska }
\def\zb{\boldsymbol{z}}
\def\Pb{\boldsymbol{P}}
\def\Sigmab{\boldsymbol{\Sigma}}

\def\Vb{\boldsymbol{V}}
\def\jb{\boldsymbol{j}}

\def\ILG{\mathtt{I}_h}
\def\IRT{\boldsymbol{\mathtt{I}}^\mathrm{RT}_h}
\def\IND{\boldsymbol{\mathtt{I}}^\mathrm{ND}_h}
\def\PND{\boldsymbol{\mathtt{\Pi}}^\mathrm{ND}_h}
\def\PLG{\mathtt{\Pi}_h}
\def\ILGav{\mathcal{I}_h}
\def\INDav{\boldsymbol{\mathcal{I}}^{\mathrm{ND}}_h}
\def\IRTav{\boldsymbol{\mathcal{I}}^{\mathrm{RT}}_h}
\def\IX{\boldsymbol{\mathcal{I}}^{\boldsymbol{\mathrm{X}}}_h}
\def\PX{\boldsymbol{\mathtt{\Pi}}^{\boldsymbol{\mathrm{X}}}}
\def\Pz{\boldsymbol{\mathtt{\Pi}}^\circ}
\def\Pzh{\boldsymbol{\mathtt{\Pi}}^\circ_h}

\def\Xb{{\boldsymbol{X}}}
\def\Yb{{\boldsymbol{Y}}}
\def\Holders{H\"{o}lder's }
\def\pstar{{q}}
\def\zerob{\boldsymbol{0}}
\def\Zb{\boldsymbol{Z}}
\def\Sigmabh{\boldsymbol{\Sigma}_h}
\def\otop{\perp\mkern-20.7mu\bigcirc\,}
\def\hform{\boldsymbol{\mathfrak{H}}}
\def\Null{\mathcal{N}}
\def\hb{\boldsymbol{h}}

\title{A Mixed Finite Element Method for the Dirichlet Vector Laplacian in Three Dimensions}
\author[]{Ralf Hiptmair\thanks{Seminar for Applied Mathematics, ETH Z\"{u}rich, \texttt{ralf.hiptmair@sam.math.ethz.ch}} }
\author[]{Peiyang Yu\thanks{Seminar for Applied Mathematics, ETH Z\"{u}rich, \texttt{peiyang.yu@sam.math.ethz.ch}} }
\author[]{Tianwei Yu\thanks{Seminar for Applied Mathematics, ETH Z\"{u}rich, \texttt{tianwei.yu@sam.math.ethz.ch}~(corresponding author)}}
\affil[]{}
\date{}

\begin{document}
\maketitle
\begin{abstract}
    This work establishes the well-posedness and \emph{a priori} error estimates for the mixed FEEC-type finite element approximation of the three-dimensional vector Laplace boundary value problem subject to Dirichlet boundary conditions. The Dirichlet condition disrupts the structure of the standard de Rham complex, requiring the vorticity to be sought in a non-standard function space to achieve well-posedness. We derive error estimates that confirm the numerically observed suboptimal convergence rates. In particular, by developing a discrete Caccioppoli-type inequality for discrete curl-harmonic functions, we prove $(k-1/2)$-th order convergence in the energy norm on general domains and $k$-th order convergence in $\Lb^2$ on convex domains, where $k \ge 1$ denotes the polynomial degree of the finite element spaces. These results extend the previous two-dimensional analysis developed in [Arnold, D.N., Falk, R.S. and Gopalakrishnan, J., 2012. \emph{Mixed finite element approximation of the vector Laplacian with Dirichlet boundary conditions}. Mathematical Models and Methods in Applied Sciences, 22(9), p.1250024.]~to three-dimensional domains with general topology. As a direct application, a discretization of the Stokes problem in vorticity-velocity-pressure form is studied.
\end{abstract}

\section{Introduction}
Let $\Omega \subset \mathbb{R}^3$ be a bounded Lipschitz polygonal domain. The vector Laplace boundary value problem (BVP) with homogeneous Dirichlet boundary conditions is
\begin{equation}\label{eq:veclap}
\begin{aligned}
    - \vecDelta \ub &= \fb &&\text{in } \Omega,\\
    \ub &= \zerob && \text{on } \partial\Omega.\\ 
\end{aligned}
\end{equation}
Using the vector calculus identity $-\vecDelta = \curl\curl - \grad \div$, the equation~\eqref{eq:veclap} can be rewritten as
\begin{align}\label{eq:veclap1st}
    \mub = \curl \ub, \quad \curl \mub - \grad \div \ub = \fb
\end{align}
where an additional unknown $\mub$ (called vorticity) is introduced. We consider the weak form of~\eqref{eq:veclap1st} seeking $(\mub, \ub) \in \Sigmab\times\Hzdiv$ such that
\begin{equation}\label{eq:hodgelap}
    \begin{aligned}
        (\mub, \taub) - \langle\curl \taub, \ub\rangle &= 0 &&\forall\,\taub \in \Sigmab, \\
        \langle\curl\mub, \vb\rangle + (\div\ub, \div\vb) &= \langle\fb, \vb\rangle && \forall\,\vb\in \Hzdiv,
    \end{aligned}
\end{equation}
where $\Sigmab$ will be defined below. In contrast to the vector Hodge Laplacian with \emph{standard boundary conditions} $\ub \cdot\nb = 0, \curl \ub \times \nb = \zerob$ or $\ub \times \nb = \zerob, \div \ub = 0$ on $\partial\Omega$~\cite[][Section~4.5]{arnold_2018}, the problem~\eqref{eq:hodgelap} is subject to the full Dirichlet condition $\ub = \zerob$ and is not amenable to straightforward Galerkin discretization in the framework of \emph{Finite Element Exterior Calculus} (FEEC)~\cite{arnold_2006,arnold_2018}.  In fact, a naive choice $\Sigmab := \Hcurl$ leads to ill-posedness. As a remedy, the authors of~\cite{dubois_2003,arnold_2012} have proposed using the following superspace of $\Hcurl$:
\begin{equation}\label{def:Sigmab}
    \Sigmab:=\{\taub\in \Lb^2: \curl \taub \in \Hzdiv'\}
\end{equation}
endowed with the norm $\|\taub\|^2_{\Sigmab}:=\|\taub\|^2 + \|\curl\taub\|^2_{\Hzdiv'}$. 

It was observed numerically that FEEC discretizations of~\eqref{eq:hodgelap} work well but suffer from half-order suboptimal convergence~\cite{dubois_2003b, arnold_2012}, which was later demonstrated theoretically by~\cite{arnold_2012} on two-dimensional convex domains. In agreement with the numerical observations, they derive $(k-1/2)$-order convergence in the $\Lb^2$ error of $\mub$ and the $\Hdiv$ error of $\ub$, as well as $k$-th order convergence in the $\Lb^2$-error of $\ub$, where $k\geq 1$ is the polynomial degree of the elements. 

Regardless of the suboptimality, the mixed formulation~\eqref{eq:hodgelap} finds its application in some scenarios. It is directly related to the \emph{Vorticity-Velocity-Pressure} (VVP) formulation of the Stokes problem~\cite{dubois_2003,dubois_2003b}, which will be revisited in Section~\ref{sec:stokes}. Furthermore, the (singularly perturbed) Darcy-Stokes-Brinkman equation~\cite{bernardi_2005,vassilevski_2014,alvarez_2016} can be treated naturally by such a formulation, while other methods entail either customized finite element spaces~\cite{mardal_2002,xie_2008,guzman_2012b} or stabilization terms~\cite{burman_2007,badia_2009,hansbo_2009,juntunen_2010}.

Although the two-dimensional version of~\eqref{eq:hodgelap} has been analyzed theoretically, the three-dimensional case has remained open. Numerical evidence~\cite{dubois_2003b} (see also Section~\ref{sec:num}) suggests a similar convergence behavior to that in two dimensions. Techniques developed in~\cite{arnold_2012} cannot be directly adapted to three dimensions. The fundamental difference lies in the fact that the curl operator reduces to a rotated gradient in two dimensions, while the vector $\curl$ is relevant in three dimensions. Some $L^p$ arguments are essential in the error analysis (see Section~\ref{sec:errest}) with $p \in [1,\infty]$. They are conveniently available for $H^1$-elliptic problems (see, e.g.~\cite[][Section~8.5]{brenner_2007}), but not for $\Hcurl$-elliptic ones.

In this work, we provide \emph{a priori} error estimates for the FEEC discretization of problem~\eqref{eq:hodgelap} in three dimensions. This entails adaptations of arguments in~\cite{arnold_2012} and some new techniques. Furthermore, additional effort is dedicated to covering domains of general topology. In overview, we extend the result of~\cite{arnold_2012} in the following respects:
\begin{itemize}
    \item extension to three dimensions;
    \item extension to domains with general topology;
    \item removal of the $|\log h|$-factor in the error bounds.
\end{itemize}

This paper is organized as follows. Section~\ref{sec:cont} establishes the well-posedness of the continuous formulation~\eqref{eq:hodgelap} by deriving an equivalent norm for the vorticity space $\Sigmab$. In Section~\ref{sec:disc}, we first provide some preliminaries needed in the analysis of the FEEC discretization~\eqref{eq:hodgelapdisc}, especially the approximation property of a curl-elliptic projection. Afterward, the stability of~\eqref{eq:hodgelapdisc} is shown, and \emph{a priori} error estimates are derived. Section~\ref{sec:stokes} applies this framework to the discretization of the Stokes problem in the VVP formulation, establishing well-posedness and \emph{a priori} error bounds for both velocity and pressure. Finally, Section~\ref{sec:num} presents results from numerical experiments to validate the theoretical convergence results.

\subsection*{Notations.}
Throughout the presentation, we adhere to the convention that variables in plain style are scalar fields, while those in bold style are vector fields, that is, $\ub = \left[u_1, u_2, u_3\right]^\top$. The same rule applies to spaces (e.g., $\Lb^2$ stands for the vector $L^2$ space) and operators. We recall the relevant differential operators:
\begin{equation*}
    \grad: u\mapsto\begin{bmatrix}\partial_x u \\ \partial_y u \\ \partial_z u\end{bmatrix},\quad\curl: \ub\mapsto\begin{bmatrix}\partial_y u_3 - \partial_z u_2\\ \partial_z u_1 - \partial_x u_3 \\ \partial_x u_2 - \partial_y u_1 \end{bmatrix},\quad \div: \ub\mapsto \partial_x u_1 + \partial_y u_2 + \partial_z u_3.
\end{equation*}
The parentheses $(\cdot,\cdot)$ stand for $L^2$ (or $\Lb^2$) inner product, while the angle brackets $\langle\cdot,\cdot\rangle$ represent duality pairings that will be clear from the context. Given $D \subset \R^3$ a bounded Lipschitz domain, we adopt the standard notations for the spaces $\Hb(\boldsymbol{\mathrm{curl}},D),  \Hb(\mathrm{div},D)$ and the (fractional) Sobolev spaces $W^{s,p}(D), s \in (0,\infty), p \in [1,\infty]$ with $H^s(D) := W^{s,2}(D)$. Denoting by $\nb$ the outer normal vector of $D$, the spaces with vanishing traces (in a suitable sense) on $\partial D$ are denoted by $H^1_0(D):=\{u \in H^1(D): u = 0 \text{ on }\partial D\}, \Hb_0(\boldsymbol{\mathrm{curl}},D) := \{\ub \in \Hb(\boldsymbol{\mathrm{curl}},D): \ub \times\nb = 0 \text{ on }\partial D\}, \Hb_0(\mathrm{div},D):=\{\ub \in  \Hb(\mathrm{div},D): \ub\cdot\nb = 0\text{ on }\partial D\}$. Additionally, we let $L^2_0(D):= \{u \in L^2(D): (u,1) = 0\}$ consist of $L^2$ functions with zero mean. We write $\|\cdot\| = \|\cdot\|_{L^2(D)}$ or $\|\cdot\|_{\Lb^2(D)}$ for brevity. The kernel spaces $\Hb(\boldsymbol{\mathrm{curl}}\boldsymbol{0},D),  \Hb(\mathrm{div}0,D)$ consist of functions that have zero curl and zero divergence, respectively. When $D = \Omega$, the spatial domain is omitted for brevity; e.g., $H^1 := H^1(\Omega)$. 
Given $\Yb\subset L^2$ (or $\Lb^2$), we denote by $\Yb^\perp$ its $L^2$-orthogonal complement. Given an operator $\mathsf{T}$, its null space is denoted by $\mathcal{N}(\mathsf{T})$. 
Throughout the presentation, we use the symbol $C$ to represent a generic positive constant that may change from line to line. Importantly, it may depend on the spatial domain, the mesh shape regularity, and the polynomial degree, but not on the mesh size. By $A \simeq B$, we mean that there exist two generic constants $C_1, C_2$ such that $C_1 A \leq B \leq C_2A$.

\section{Well-posedness of the continuous formulation}\label{sec:cont}

\subsection{Preliminaries}
\subsubsection{The \texorpdfstring{$\Lb^2$}{L2} de Rham complex}
In this section, we briefly recap the Hilbert complex structure of the $\Lb^2$ de Rham complex, which provides the functional analytic tools to study~\eqref{eq:hodgelap}. We follow the notations and definitions used in~\cite[][Chapter~4]{arnold_2018}. The primal $\Lb^2$ de Rham complex is
\begin{equation}\label{eq:derhamcomplexprimal}
    \begin{tikzcd}[row sep=0.5em, column sep=2em]
        0 \arrow[r]
        & H^1 \arrow[r, "\grad"]
        & \Hcurl \arrow[r, "\curl"]
        & \Hdiv \arrow[r, "\div"]
        & L^2 \arrow[r]
        & 0
    \end{tikzcd}
\end{equation}
and its dual complex
\begin{equation}\label{eq:derhamcomplexdual}
    \begin{tikzcd}[row sep=0.5em, column sep=2em]
        0
        & L^2 \arrow[l]
        & \Hzdiv \arrow[l, "-\div"']
        & \Hzcurl \arrow[l, "\curl"']
        & H^1_0 \arrow[l, "-\grad"']
        & 0 \arrow[l].
    \end{tikzcd}
\end{equation}
We have the Hodge decompositions~\cite[][Theorem~4.5]{arnold_2018} arising from the Hilbert complex structure:
\begin{align}
    L^2 &=\hform^0 \otop \div \Hzdiv , \label{def:hodgedecom0}\\
    \Lb^2 &= \overbrace{\grad H^1 \otop \phantom{\hform^1\,\,}}^{=\Hcurlz}\hspace{-0.5cm}\underbrace{\hform^1 \otop \curl\Hzcurl}_{=\Hzdivz}, \label{def:hodgedecom1}\\
    \Lb^2 &= \overbrace{\curl \Hcurl \otop \phantom{\hform^2\,\,}}^{=\Hdivz}\hspace{-0.5cm}\underbrace{\hform^2 \otop \grad H^1_0}_{=\Hzcurlz},\label{def:hodgedecom2} \\
    L^2 &= \div \Hdiv \label{def:hodgedecom3}
\end{align}
where $\otop$ indicates an $\Lb^2$-orthogonal sum, and $\hform^k, k = 0,1,2$, denotes the space of harmonic $k$-forms. They are characterized by
\begin{align}
    \hform^0 &= \{u \in H^1: \grad u =0\text{ in } \Omega\}, \label{def:h0}\\
    \hform^1 &= \{\ub \in \Hcurl\cap \Hzdiv: \curl\ub = \zerob, \div \ub = 0 \text{ in } \Omega\}, \label{def:h1}\\
    \hform^2 &= \{\ub \in \Hzcurl \cap \Hdiv: \curl\ub = \zerob, \div \ub = 0 \text{ in } \Omega\}. \label{def:h2}
\end{align}
The dimension of $\hform^k, k=0,1,2$, is equal to the $k$-th Betti number $b_k$ of the domain $\Omega$. In particular, $b_0$ is the number of connected components; $b_1$ is the number of tunnels through $\Omega$; $b_2$ is the number of voids in $\Omega$.

The vector Laplace BVP $-\vecDelta \ub = \fb$ with standard boundary conditions, namely, $\ub\times\nb = \zerob, \div \ub = 0$ or $\ub\cdot\nb = 0, \curl\ub\times\nb=\zerob$ on $\partial\Omega$, arises as a Hodge Laplace BVP of the $\Lb^2$ de Rham complex. However, the system~\eqref{eq:hodgelap}, subject to the full Dirichlet condition $\ub = \zerob$, cannot be formulated as a standard Hodge Laplace BVP. It breaks the de Rham complex structure in the sense that the system~\eqref{eq:hodgelap} is ill-posed if the function spaces are chosen as those in the de Rham complex.

In the sequel, the domain $\Omega$ will comply with a very mild assumption:
\begin{assumption}[connected domain]
    $\Omega \in \R^3$ is a bounded Lipschitz polygonal domain with only one connected component.
\end{assumption}
\noindent It is easy to see that $\hform^0 \cong \R$ holds under the assumption. Yet nontrivial harmonic functions in $\hform^1$ and $\hform^2$ may exist.

\subsubsection{Some decompositions}
In the following, we show an elementary but crucial lemma on the existence of scalar/vector potentials for certain functions (see, e.g.,~\cite[][Proposition~2.1]{martin_2025}).
\begin{lemma}[scalar/vector potentials]\label{thm:potscalarvec}
\begin{align}
    \Hzcurlz &\subset \grad H^1, \label{eq:potscalar}\\
    \Hzdivz &\subset \curl\Hcurl. \label{eq:potvec}
\end{align}
\end{lemma}
\begin{proof}
    We only show~\eqref{eq:potscalar} while~\eqref{eq:potvec} follows similarly. Let $\tilde{\Omega} \supset \Omega$ be an open ball containing $\Omega$. Given $\taub \in \Hzcurlz$, denote by $\tilde{\taub} \in \boldsymbol{H}_0(\boldsymbol{\mathrm{curl}}\boldsymbol{0}; \tilde{\Omega})$ the zero extension of $\taub$ to $\tilde{\Omega}$. Notice that $\tilde{\Omega}$ has no void ($\hform^2 = \{\zerob\}$). We have by~\eqref{def:hodgedecom2} that $\grad H^1_0(\tilde{\Omega}) = \boldsymbol{H}_0(\boldsymbol{\mathrm{curl}}\boldsymbol{0}; \tilde{\Omega})$. Hence, we can find a $\tilde{w} \in H^1_0(\tilde{\Omega})$ such that $\tilde{\taub} = \grad \tilde{w}$. Clearly, $\taub = \grad \tilde{w}|_\Omega \in \grad H^1$.
\end{proof}
\noindent Combining the lemma with the decomposition~\eqref{def:hodgedecom1} and~\eqref{def:hodgedecom2} yields the next corollary.
\begin{corollary}[Helmholtz decomposition]\label{thm:helmdecomp} The following holds:

    \noindent $(i)$ given $\taub \in \Hzcurl$, there exist $\phi \in H^1, \psib \in \Hcurl$ and a positive constant $C$ such that 
    \begin{equation}\label{eq:helmdecomp1}
        \taub = \grad \phi + \curl\psib,\quad\quad (\grad \phi, \curl\psib) = 0,
        \quad\quad \|\curl \psib\| \leq C\|\curl \taub\|;
    \end{equation}

    \noindent $(ii)$ given $\vb \in \Hzdiv$, there exist $\phi \in H^1, \psib \in \Hcurl$ and a positive constant $C$ such that 
    \begin{equation}\label{eq:helmdecomp2}
        \vb = \grad \phi + \curl\psib,\quad\quad (\grad \phi, \curl\psib) = 0,
        \quad\quad \|\grad \phi\| \leq C\|\div \vb\|.
    \end{equation}
\end{corollary}
Lastly, we recall the well-known regular decomposition of $\Hzdiv$.
\begin{theorem}[regular decomposition of $\Hzdiv$\text{~\cite[][Theorem~6]{hiptmair_2020}}]\label{thm:regdecomphdiv}
    For any $\vb \in \Hzdiv$, there exists $\hat{\vb}, \hat{\taub} \in \Hb^1_0$ depending linearly on $\vb$ such that
        $\vb = \hat{\vb} + \curl \hat{\taub}$
    and 
    \begin{equation}\label{eq:regdecomphdiv}
        \begin{aligned}
            \|\hat{\vb}\| + \|\hat{\taub}\|_{\Hcurl} \leq C\|\vb\|, \quad\quad\|\hat{\vb}\|_{\Hb^1} + \|\hat{\taub}\|_{\Hb^1} \leq C\|\vb\|_{\Hdiv}.
        \end{aligned}
    \end{equation}
\end{theorem}

\subsection{Norm equivalence on \texorpdfstring{$\Sigmab$}{Sigma}}\label{sec:normeq}
Before showing the well-posedness of~\eqref{eq:hodgelap}, we state an equivalent norm on the space $\Sigmab$ (see its definition in~\eqref{def:Sigmab}). 

First, we introduce a projection $\Pz: \Sigmab \rightarrow \Hzcurl\cap \Hzcurlz^\perp$ such that $\taub_0 := \Pz\taub$ solves
\begin{equation}\label{def:tau0}
    (\curl \taub_0, \curl \psib) = \langle \curl \taub, \curl \psib\rangle \quad\forall\,\psib\in \Hzcurl.
\end{equation}
The \Poincare inequality~\cite[][Theorem~4.6]{arnold_2018} and the Lax-Milgram theorem (see, e.g.,~\cite[][Lemma~25.2]{ern_2021b}) imply that $\taub_0$ is uniquely defined.
\begin{lemma}[norm equivalence on $\Sigmab$]\label{thm:sigmanormequiv}
    For any $\taub \in \Sigmab$, it holds that
    \begin{equation}
        \|\taub\|_{\Sigmab} \simeq \|\taub\| + \|\Pz\taub\|_{\Hcurl}
    \end{equation}
    where $\Pz:\Sigmab\rightarrow \Hzcurl\cap \Hzcurlz^\perp$ is defined in~\eqref{def:tau0}. 
\end{lemma}
\begin{proof}
    Using the regular decomposition of $\Hzdiv$ in Theorem~\ref{thm:regdecomphdiv}, we deduce that
    \begin{equation}
        \begin{aligned}
            \|\curl\taub\|_{\Hzdiv'} &= \sup_{\zerob \neq \vb \in \Hzdiv} \frac{\langle \curl \taub, \vb\rangle}{\|\vb\|_{\Hzdiv} } = \sup_{\zerob \neq \vb \in \Hzdiv} \frac{\langle \curl \taub, \hat{\vb} + \curl \hat{\taub} \rangle}{\|\vb\|_{\Hzdiv} } \\
            &= \sup_{\zerob \neq \vb \in \Hzdiv} \frac{(\taub, \curl\hat{\vb}) + (\curl \Pz\taub, \curl \hat{\taub} )}{\|\vb\|_{\Hzdiv} } \\
            &\leq C(\|\taub\| + \|\curl\Pz \taub\|).
        \end{aligned}
    \end{equation}
    It is clear from~\eqref{def:tau0} that $\|\curl\Pz\taub\| \leq \|\curl \taub\|_{\Hzdiv'}$, and we can conclude.
\end{proof}
\begin{remark}[comparison with~\cite{arnold_2012}]
    We provide a more direct proof of the $ \Sigmab $-norm equivalence than that in \cite[][Section~3.2]{arnold_2012}. It has the added advantage of naturally accommodating domains with nontrivial topology. The same remark applies to the discrete counterpart of Lemma~\ref{thm:sigmanormequiv}~(see Lemma~\ref{thm:sigmanormequivdisc}).
\end{remark}

\subsection{Well-posedness}
Denote the bilinear form associated with \eqref{eq:hodgelap} by
\begin{align}
    \a((\mub,\ub),(\taub,\vb)) &:= (\mub,\taub) + (\div\ub,\div\vb) - \langle\curl\taub,\ub\rangle + \langle\curl\mub,\vb\rangle.\label{def:a}
\end{align}
We can rewrite~\eqref{eq:hodgelap} as seeking $(\mub,\ub) \in \Sigmab\times\Hzdiv$ such that
\begin{equation}\label{eq:hodgelaprd}
    \a((\mub,\ub),(\taub,\vb)) = \langle\fb,\vb\rangle \quad \forall\,(\taub, \vb)\in\Sigmab\times\Hzdiv.
\end{equation}
The following theorem establishes the well-posedness of~\eqref{eq:hodgelap} (or~\eqref{eq:hodgelaprd}).
\begin{theorem}[well-posedness]\label{thm:well-posed}
     The problem~\eqref{eq:hodgelap} admits a unique solution $(\mub,\ub)\in \Sigmab\times\Hzdiv$. There exists a positive constant $C$ independent of $\fb \in \Hzdiv'$ such that
     \begin{equation}
         \|\mub\|_{\Sigmab} + \|\ub\|_{\Hdiv} \leq C \|\fb\|_{\Hzdiv'}.
     \end{equation}
\end{theorem}
\begin{proof}
    By the BNB theorem~\cite[][Theorem~25.9]{ern_2021b}, it suffices to show the boundedness and inf-sup condition of the bilinear form $\a$, which are demonstrated in the next lemma.
\end{proof}
\begin{lemma}[boundedness and inf-sup condition]\label{thm:bddinfsup}
    There exist positive constants $C, \beta_\a$ such that
    \begin{align}
         \a((\mub,\ub),(\taub,\vb)) \leq C \|(\mub,\ub)\|_{\Sigmab\times\Hdiv}\|(\taub,\vb)\|_{\Sigmab\times\Hdiv}&  && \forall\,(\mub,\ub), (\taub,\vb)\in \Sigmab\times\Hzdiv, \label{eq:bddacont}\\
        \sup_{0\neq (\taub, \vb)\in\Sigmab\times\Hzdiv}\frac{\a((\mub,\ub),(\taub,\vb))}{\|(\taub,\vb)\|_{\Sigmab\times\Hdiv}} \geq \beta_\a\|(\mub,\ub)\|_{\Sigmab\times\Hdiv}& && \forall\,(\mub,\ub)\in \Sigmab\times\Hzdiv.\label{eq:infsupa}
    \end{align}
\end{lemma}
\begin{proof}
    The estimate~\eqref{eq:bddacont} is obvious. We show~\eqref{eq:infsupa}. Fix $(\mub, \ub) \in \Sigmab \times \Hzdiv$. By the Helmholtz decomposition~\eqref{eq:helmdecomp2}, we have $\ub = \grad\phi_\ub + \curl\psib_\ub $ where $\phi_\ub\in H^1/\R, \psib_\ub\in\Hcurl \cap \Hcurlz^\perp$ satisfy the orthogonality $(\grad\phi_\ub, \curl\psib_\ub) = 0$ and the \Poincare inequalities~\cite[][Theorem~4.6]{arnold_2018}:
    \begin{align}
        \|\psib_\ub\| \leq C_p \|\curl\psib_\ub\|, \quad \|\grad \phi_\ub\| \leq c_p\|\div \ub\|.\label{eq:poincare}
    \end{align}
    Suitable $(\taub,\vb)\in\Sigmab\times\Hzdiv$ for the given $(\mub,\ub)\in\Sigmab\times\Hzdiv$ is to be chosen.
    We have 
    \begin{equation}\label{eq:thminfsup-1}
    \begin{aligned}
        \a((\mub,\ub),(\mub,\ub)) &= \|\mub\|^2 + \|\div\ub\|^2, \\
        \a((\mub,\ub),(\zerob,\curl\Pz\mub)) &= \langle\curl\mub,\curl\Pz\mub\rangle = \|\curl \Pz\mub\|^2,\\
        \a((\mub,\ub),(-\psib_\ub,\zerob)) &= -(\mub,\psib_\ub) + (\curl \psib_\ub,\ub) \\
        &\geq - \frac{1}{2\gamma}\|\mub\|^2 - \frac{\gamma}{2}\|\psib_\ub\|^2 + \|\curl\psib_\ub\|^2 \\
        &\geq - \frac{1}{2\gamma}\|\mub\|^2 + (1 - \frac{\gamma}{2}C_p)\|\curl\psib_\ub\|^2 \\
        &= - \frac{C_p}{2}\|\mub\|^2 + \frac{1}{2}\|\curl\psib_\ub\|^2. \quad (\text{set } \gamma = C_p^{-1})
    \end{aligned}
    \end{equation}
    Let $\taub = \mub - \delta\psib_\ub$ and $\vb = \ub + \curl\Pz\mub$. By~\eqref{eq:thminfsup-1}\eqref{eq:poincare} and the $\Sigmab$-norm equivalence (see Lemma~\ref{thm:sigmanormequiv}), we have
    \begin{equation}
        \begin{aligned}
            \a((\mub,\ub),(\taub,\vb)) &\geq \left(1-\frac{C_p}{2}\delta\right)\|\mub\|^2 + \|\curl\Pz\mub\|^2 + \frac{1}{2}\delta\|\curl\psib_\ub\|^2 + \|\div\ub\|^2 \\
            &\geq \left(1-\frac{C_p}{2}\delta\right)\|\mub\|^2 + \|\curl\Pz\mub\|^2\\
            &\quad\quad + \frac{1}{2}\delta\|\curl\psib_\ub\|^2 + \frac{1}{2}\delta\|\grad\phi_\ub\|^2 + \left(1 - \frac{c_p\delta}{2}\right)\|\div\ub\|^2 \\
            &\geq \left(1-\frac{C_p}{2}\delta\right)\|\mub\|^2 + \|\curl\Pz\mub\|^2\\
            &\quad\quad + \frac{1}{2}\delta\|\ub\|^2 + \left(1 - \frac{c_p\delta}{2}\right)\|\div\ub\|^2 \\
            &\geq C(\|\mub\|^2_{\Sigmab} + \|\ub\|^2_{\Hdiv}), \quad (\text{suitably choose } \delta) \\
            \|\taub\|_{\Sigmab} + \|\vb\|_{\Hdiv} &\leq C(\|\mub\|_{\Sigmab} + \|\ub\|_{\Hdiv}),
        \end{aligned}
    \end{equation}
    which leads to~\eqref{eq:infsupa}.
\end{proof}
\begin{remark}[comparison with the standard Hodge Laplacian]
    It is instructive to compare the derivation of the inf-sup condition~\eqref{eq:infsupa} with that of its counterpart~\cite[][Theorem~4.9]{arnold_2018} for the standard boundary condition $\ub \cdot \nb = 0, \curl \mub \times \nb = \zerob$. Specifically, the standard case takes $\Sigmab = \Hzcurl$, and a new unknown in the harmonic $2$-form $\hform^2$ has to be introduced. 
\end{remark}

\section{Discretization}\label{sec:disc}
Let $\{\mathcal{T}_h\}$ be a \emph{shape-regular} and \emph{quasi-uniform} family of \emph{simplicial} triangulations of $\Omega \subset \R^3$ with $h$ standing for the (typical) mesh width. Let $k \in \mathbb{N}$ be a (fixed) polynomial degree. We write $\mathcal{P}_k$ for the space of $3$-variate polynomials of total degree $\leq k$ and $\boldsymbol{\mathcal{P}}_k = \mathcal{P}^3_k$. The following canonical finite element spaces are considered:
\begin{equation}\label{def:spaces}
    \begin{aligned}
        P_h &:= \{v \in H^1:v|_T \in \mathcal{P}_k\;\;\forall\,T\in\mathcal{T}_h\}, & P_{h,0} &:= P_h\cap H^1_0, \\
        \Sigmabh &:= \{\vb \in \Hcurl: \vb|_T \in \boldsymbol{\mathcal{P}}_{k-1} + \xb\times\boldsymbol{\mathcal{P}}_{k-1}\;\;\forall\,T\in\mathcal{T}_h\}, & \Sigmab_{h,0} &:= \Sigmab_{h}\cap\Hzcurl,\\
        \Vb_{h} &:= \{\vb \in \Hdiv: \vb|_T \in \boldsymbol{\mathcal{P}}_{k-1} + \xb\mathcal{P}_{k-1}\;\;\forall\,T\in\mathcal{T}_h\},  & \Vb_{h,0} &:= \Vb_h\cap\Hzdiv, \\
        S_h &:= \{v \in L^2: v|_T \in \mathcal{P}_{k-1}\;\;\forall\,T\in\mathcal{T}_h\}, & S_{h,0} &:= S_h\cap L^2_0,
    \end{aligned}
\end{equation}
representing the Lagrange element, the \Nedelec element, the Raviart-Thomas element, and the discontinuous element. In the notation of~\cite{arnold_2018}, they correspond to $\mathcal{P}^{-}_k\Lambda^l(\mathcal{T}_h),\,l= 0,1,2,3$, respectively. 

We consider the discretization of~\eqref{eq:hodgelap} that seeks $(\mub_h, \ub_h) \in \Sigmabh \times \Vb_{h,0}$ such that
\begin{equation}\label{eq:hodgelapdisc}
\begin{aligned}
    (\mub_h, \taub_h) - (\curl \taub_h, \ub_h) &= 0 &&\forall\,\taub_h \in \Sigmabh, \\
    (\curl\mub_h, \vb_h) + (\div\ub_h, \div\vb_h) &= \langle\fb, \vb_h\rangle && \forall\,\vb_h\in \Vb_{h,0}.
\end{aligned}
\end{equation}
In view of the bilinear form $\a$ defined in~\eqref{def:a}, the system~\eqref{eq:hodgelapdisc} can be rewritten as 
\begin{equation}\label{eq:hodgelapdiscrd}
    \a((\mub_h,\ub_h),(\taub_h,\vb_h)) = \langle\fb,\vb_h\rangle \quad \forall\,(\taub_h, \vb_h)\in\Sigmabh\times\Vb_{h,0}.
\end{equation}
In this section, we establish stability and \emph{a priori} error estimates for~\eqref{eq:hodgelapdisc}.
\subsection{Preliminaries}

\subsubsection{Regularity results}
We first recall some well-known regularity results on Lipschitz polygonal domains.
\begin{theorem}[regularity lift of scalar Poisson's equation\text{~\cite[][Corollary~2.6.7]{grisvard_2011}}]\label{thm:reglap}
    Let $D \subset \R^3$ be a bounded Lipschitz polygonal domain. Let $u \in H^1_0(D)$ or $H^1(D)/\R$ be the weak solution of $-\Delta u =f$ with $f \in L^2(D)$. It holds that $u \in H^s(D)$ and there exists a positive constant $C$ independent of $f$ such that $\|u\|_{H^s(D)} \leq C\|f\|$ where $s \in (3/2, 2]$ depends on the geometry of $D$. 
\end{theorem}
\begin{theorem}[embedding of $\Hzcurl\cap \Hdiv$ into $\Hb^s$\text{~\cite[][Proposition~3.7]{amrouche_1998}}]\label{thm:embdd}
    Let $D \subset \R^3$ be a bounded Lipschitz polygonal domain. The continuous injection $\Hb_0(\mathrm{curl},D)\cap\Hb(\mathrm{div},D) \hookrightarrow \Hb^s(D)$ holds where $s \in (1/2, 1]$ depends on the geometry of $D$. In particular, it holds that $\hform^2 \subset \Hb^s$.
\end{theorem}
\begin{proof}
    Here we show this result as a corollary of the regular decomposition (see, e.g.,~\cite[][Theorem~2.1]{hiptmair_2019}). For each $\ub \in \Hzcurl$, there exist $\zb \in \Hb^1_0(D)$ and $\varphi \in H^1_0(D)$ such that $\ub = \zb + \grad \varphi$. Since we assume that $\ub \in \Hdiv$, it holds that $-\Delta \varphi = \div(\zb-\ub) \in L^2(D)$. By Theorem~\ref{thm:reglap}, we obtain lifted regularity $\varphi \in H^{s+1}$ with $s \in (1/2,1]$. Hence, $\ub \in \Hb^{s}$, and one can show that $\|\ub\|_{\Hb^s} \leq C(\|\ub\|_\Hcurl + \|\ub\|_{\Hdiv})$.
\end{proof}

\subsubsection{Discrete de Rham complexes and interpolation approximations}\label{sec:interapprox}
The discrete spaces introduced in~\eqref{def:spaces} form sub-complexes (see~\eqref{eq:derhamcomplexprimaldisc}) of the $\Lb^2$ de Rham complexes~\eqref{eq:derhamcomplexprimal}\eqref{eq:derhamcomplexdual}. We denote by 
\begin{equation}
\ILG: \tilde{H}^1 \rightarrow P_h,\quad\IND:\tilde{\Hb}(\boldsymbol{\mathrm{curl}})\rightarrow \Sigmabh,\quad\IRT:\tilde{\Hb}(\mathrm{div})\rightarrow \Vb_h,\quad\PLG:L^2\rightarrow S_h
\end{equation}
the \emph{canonical interpolation} operators that are also projections. Here, $\tilde{H}^1 \subset H^1, \tilde{\Hb}(\boldsymbol{\mathrm{curl}})\subset\Hcurl, \tilde{\Hb}(\mathrm{div})\subset \Hdiv$ are subspaces with suitable regularity (see~\cite[][Theorem~5.3]{arnold_2006}) on which these canonical interpolations are well-defined. Importantly, the following diagram commutes:

\begin{equation}\label{eq:commuting}
    \begin{tikzcd}[row sep=1.5em, column sep=3em]
        0 \arrow[r]
        & \tilde{H}^1 \arrow[d, "\ILG"] \arrow[r, "\grad"]
        & \tilde{\Hb}(\boldsymbol{\mathrm{curl}}) \arrow[d, "\IND"] \arrow[r, "\curl"]
        & \tilde{\Hb}(\mathrm{div}) \arrow[d, "\IRT"] \arrow[r, "\div"]
        & L^2 \arrow[d, "\PLG"] \arrow[r]
        & 0\\
        0 \arrow[r]
        & P_{h} \arrow[r, "\grad"]
        & \Sigmab_{h} \arrow[r, "\curl"]
        & \Vb_{h} \arrow[r, "\div"]
        & S_{h} \arrow[r]
        & 0 
    \end{tikzcd}
\end{equation}
The counterpart of~\eqref{eq:commuting} with homogeneous boundary conditions holds as well.

Recall the approximation properties of $\IND$ and $\IRT$~\cite[][Theorem~5.3]{arnold_2006}: for $r, p \in \R$, it holds that whenever the right-hand sides are meaningful,
    \begin{equation}\label{eq:canonicalapp}
        \begin{aligned}
            \|\vb - \IND\vb\|_{\Lb^p} &\leq Ch^{r}|\vb|_{\Wb^{r,p}}, & 2/p < \hspace{-0.25cm}&&r\leq k, \quad 1<p\leq\infty,\\
            \|\curl(\vb - \IND\vb)\|_{\Lb^p} &\leq Ch^{r}|\curl\vb|_{\Wb^{r,p}}, &1/p < \hspace{-0.25cm}&&r\leq k,\quad 1\leq p\leq \infty,\\
            \|\vb - \IRT\vb\|_{\Lb^p} &\leq Ch^{r}|\vb|_{\Wb^{r,p}}, &1/p < \hspace{-0.25cm}&&r \leq k,\quad 1\leq p \leq \infty,\\
            \|\div(\vb - \IRT \vb)\|_{L^p} &\leq Ch^{r}|\div\vb|_{W^{r,p}}, &0\leq \hspace{-0.25cm}&&r \leq k,\quad 1\leq p\leq\infty.\\
        \end{aligned}
    \end{equation}
Notice that all these estimates hold locally on each element $T \in \mathcal{T}_h$.
\begin{remark}[commuting $\Lb^2$-stable quasi-interpolations]\label{rmk:quasiint}
 The domains of the \emph{canonical interpolations} $\ILG, \IND,\IRT$ entail certain regularity. There exist \emph{quasi-interpolation} operators (see, e.g.,~\cite[][Theorem~5.6]{arnold_2006} or~\cite[][Chapter~23]{ern_2021}), denoted by $\ILGav, \INDav, \IRTav$, that are $\Lb^2$-stable, have optimal \emph{global} approximation properties, and commute with the differential operators $\grad,\curl,\div$. In the FEEC framework, these quasi-interpolation operators constitute a \emph{bounded cochain projection}. Throughout the presentation, we mostly rely on $\ILG, \IND,\IRT$, while only Lemma~\ref{thm:curlproj} and~\ref{thm:sigmanormequivdisc} entail quasi-interpolations.
\end{remark}

We have the discrete de Rham complexes as subcomplexes of the $\Lb^2$ de Rham complex~\eqref{eq:derhamcomplexprimal}\eqref{eq:derhamcomplexdual}:
\begin{equation}\label{eq:derhamcomplexprimaldisc}
    \begin{tikzcd}[row sep=0.5em, column sep=1.5em]
        0 \arrow[r]
        & P_{h} \arrow[r, "\grad"]
        & \Sigmabh \arrow[r, "\curl"]
        & \Vb_{h} \arrow[r, "\div"]
        & S_{h} \arrow[r]
        & 0, \quad\quad
        0 \arrow[r]
        & P_{h,0} \arrow[r, "\grad"]
        & \Sigmab_{h,0} \arrow[r, "\curl"]
        & \Vb_{h,0} \arrow[r, "\div"]
        & S_{h} \arrow[r]
        & 0,
    \end{tikzcd}
\end{equation}
which induce the dual complexes respectively:
\begin{equation}\label{eq:derhamcomplexdualdisc}
    \begin{tikzcd}[row sep=0.5em, column sep=1.5em]
        0
        & P_h \arrow[l]
        & \Sigmabh \arrow[l, "-\divh"']
        & \Vb_h \arrow[l, "\curlh"']
        & S_h \arrow[l, "-\gradh"']
        & 0 \arrow[l],\quad\quad
        0
        & P_{h,0} \arrow[l]
        & \Sigmab_{h,0} \arrow[l, "-\divhz"']
        & \Vb_{h,0} \arrow[l, "\curlhz"']
        & S_h \arrow[l, "-\gradhz"']
        & 0 \arrow[l].
    \end{tikzcd}
\end{equation}
The adjoint operators are characterized as follows:
\begin{equation}\label{eq:discadj}
\begin{aligned}
    \divh:&& \Sigmabh &\rightarrow P_h, &(-\divh \taub_h, w_h) &:= (\taub_h, \grad w_h) &&\forall\, w_h\in P_h, \\
    \divhz:&& \Sigmab_{h,0} &\rightarrow P_{h,0},  &(-\divhz \taub_h, w_h) &:= (\taub_h, \grad w_h) &&\forall\, w_h\in P_{h,0}, \\
    \curlh:&&\Vb_h &\rightarrow \Sigmab_h, &(\curlh \vb_h, \taub_h) &:= (\vb_h, \curl \taub_h) &&\forall\,\taub_h \in \Sigmabh, \\
    \curlhz: &&\Vb_{h,0} &\rightarrow \Sigmab_{h,0}, &(\curlhz \vb_h, \taub_h) &:= (\vb_h, \curl \taub_h) &&\forall\, \taub_h \in \Sigmab_{h,0},\\
    \gradh: &&S_h &\rightarrow \Vb_h, &(-\gradh q_h, \vb_h) &:= (q_h, \div \vb_h) &&\forall\, \vb_h \in \Vb_h, \\
    \gradhz: &&S_h &\rightarrow \Vb_{h,0}, &(-\gradhz q_h, \vb_h) &:= (q_h, \div \vb_h) &&\forall\, \vb_h \in \Vb_{h,0}.
\end{aligned}
\end{equation}
Note that the discrete de Rham complexes~\eqref{eq:derhamcomplexprimaldisc}, with or without homogeneous boundary conditions, must be treated differently and induce two different Hodge decompositions (see~\eqref{def:hodgedecomdisc0}-\eqref{def:hodgedecomdisc3}). In contrast, the boundary conditions naturally appear in the dual $\Lb^2$ de Rham complex~\eqref{eq:derhamcomplexdual}.

Being Hilbert complexes,~\eqref{eq:derhamcomplexprimaldisc}\eqref{eq:derhamcomplexdualdisc} give rise to the following discrete Hodge decompositions with or without boundary constraints:
\begin{align}
    P_h &=\hform^0_h \otop \divh \Sigmabh, & P_{h,0} &=\divhz \Sigmab_{h,0}, \label{def:hodgedecomdisc0}\\
    \Sigmabh &= \overbrace{\grad P_h \otop \phantom{\hform^1_h\,\,}}^{=\Sigmabh\cap\Hcurlz}\hspace{-0.6cm}\underbrace{\hform^1_h \otop \curlh\Vb_h}_{=\Null(\divh)}, &\Sigmab_{h,0} &= \overbrace{\grad P_{h,0} \otop \phantom{\hform^1_{h,0}\,\,}}^{=\Sigmab_{h,0}\cap\Hcurlz}\hspace{-0.8cm}\underbrace{\hform^2_{h,0} \otop \curlhz\Vb_{h,0}}_{=\Null(\divhz)}, \label{def:hodgedecomdisc1}\\
    \Vb_{h} &= \overbrace{\curl \Sigmab_{h} \otop \phantom{\hform^2_{h}\,\,}}^{=\Vb_h\cap\Hdivz}\hspace{-0.6cm}\underbrace{\hform^2_h \otop \gradh S_h}_{=\Null(\curlh)}, & \Vb_{h,0} &= \overbrace{\curl \Sigmab_{h,0} \otop \phantom{\hform^2_{h,0}\,\,}}^{=\Vb_{h,0}\cap\Hdivz}\hspace{-0.8cm}\underbrace{\hform^1_{h,0} \otop \gradhz S_{h}}_{=\Null(\curlhz)}, \label{def:hodgedecomdisc2} \\
    S_h &= \div \Vb_h, &S_h  &= \hform^0_{h,0} \otop \div \Vb_{h,0}.\label{def:hodgedecomdisc3}
\end{align}
Since the discrete complexes in~\eqref{eq:derhamcomplexprimaldisc} are subcomplexes of the $\Lb^2$ de Rham complexes~\eqref{eq:derhamcomplexprimal}\eqref{eq:derhamcomplexdual} admitting a bounded cochain projection (see Remark~\ref{rmk:quasiint}), the harmonic form spaces are isomorphic to their discrete counterparts~\cite[][Theorem~5.1]{arnold_2012}. Hence, we have
\begin{equation}
    \hform^k \simeq \hform^k_h \simeq \hform^k_{h,0},\quad k = 0, 1,2.
\end{equation}
\begin{remark}[eliminating the unknown $\mub_h$]
    With the introduction of the discrete adjoint differential operators in~\eqref{eq:discadj}, the discrete problem~\eqref{eq:hodgelapdisc} is equivalent to seeking $\ub_h \in \Vb_{h,0}$ such that
    \begin{equation}\label{eq:vvpdiscsc}
        (\curlh \ub_h, \curlh \vb_h) + (\div \ub_h, \div \vb_h) = \langle\fb,\vb_h\rangle\quad\forall\,\vb_h \in \Vb_{h,0}.
    \end{equation}
\end{remark}

\subsubsection{Some discrete decompositions}
In this section, we first show the discrete counterpart of Lemma~\ref{thm:potscalarvec} and Corollary~\ref{thm:helmdecomp}.
\begin{lemma}[discrete scalar/vector potentials]\label{thm:potscalarvecdisc}
\begin{align}
    \Sigmab_{h,0} \cap \Hcurlz &\subset \grad P_h, \label{eq:potscalardisc}\\
    \Vb_{h,0}\cap\Hdivz &\subset \curl \Sigmabh. \label{eq:potvecdisc}
\end{align}
\end{lemma}
\begin{proof}
    We only show~\eqref{eq:potscalardisc} while~\eqref{eq:potvecdisc} follows similarly. Let $\tilde{\Omega} \supset \Omega$ be a convex polygonal domain containing $\Omega$ and partitioned by $\tilde{\mathcal{T}}_h$ extending $\mathcal{T}_h$. Given $\taub_h \in \Sigmab_h$, denote by $\tilde{\taub}_h \in \Sigmab_h(\tilde{\mathcal{T}}_h)$ the zero extension of $\taub_h$ to $\tilde{\Omega}$. Notice that $\tilde{\Omega}$ has no void ($\hform^2_{h,0} = \{\zerob\}$). We have by~\eqref{def:hodgedecomdisc2} that $\grad P_{h,0}(\tilde{\mathcal{T}}_h) = \Sigmab_{h,0}(\tilde{\mathcal{T}}_h) \cap \Hcurlz(\tilde{\mathcal{T}}_h)$. Hence, we can find a $\tilde{w}_h \in P_{h,0}(\tilde{\mathcal{T}}_h)$ such that $\tilde{\taub}_h = \grad \tilde{w}_h$. Clearly, $\taub_h = \grad \tilde{w}_h|_\Omega \in \grad P_h$.
\end{proof}
\noindent Combining the previous lemma with the decompositions~\eqref{def:hodgedecomdisc1} and~\eqref{def:hodgedecomdisc2} and the discrete \Poincare inequality~\cite[][Theorem~4.6]{arnold_2018} yields the next corollary.
\begin{corollary}[discrete Helmholtz decomposition]\label{thm:helmdecompdisc} The following holds:

    \noindent $(i)$ given $\taub_h \in \Sigmab_{h,0}$, there exist $w_{h} \in P_h$ and $\vb_{h} \in \Vb_{h,0}$ such that 
    \begin{equation}\label{eq:helmdecompdisc1}
        \taub_{h} = \curlhz \vb_{h} + \grad w_{h},\quad\quad (\curlhz \vb_{h}, \grad w_{h}) = 0,\quad\quad \|\curlhz \vb_h\|\leq C\|\curl \taub_h\|;
    \end{equation}
    
    \noindent $(ii)$ given $\vb_h \in \Vb_{h,0}$, there exist $\psib_{h} \in \Sigmab_h$ and $\phi_{h} \in S_{h}$ such that
    \begin{equation}\label{eq:helmdecompdisc2}
        \vb_h = \curl\psib_{h} + \gradhz \phi_{h} ,\quad\quad(\curl\psib_{h}, \gradhz \phi_{h}) = 0, \quad\quad \|\gradhz \phi_h\| \leq C\|\div \vb_h\|.
    \end{equation}
\end{corollary}
The next lemma states a regular decomposition on the local \Nedelec element space 
\begin{equation}
    \Sigmab_k(T) := \boldsymbol{\mathcal{P}}_{k-1}(T) + \xb\times\boldsymbol{\mathcal{P}}_{k-1}(T), \quad T\in \mathcal{T}_h,\,k\in\mathbb{N}.
\end{equation}
Recall that the global space $\Sigmabh = \{\vb \in \Hcurl: \vb|_T \in \Sigmab_k(T)\;\forall\,T\in\mathcal{T}_h\}$.
\begin{lemma}[local regular decomposition of $\Sigmab_k(T)$]\label{thm:lclregdecomp}
    Given $T \in \mathcal{T}_h$, for any $\xib_h \in \Sigmab_k(T)$ there exist $\hat{\xib}_h \in \Sigmab_k(T)$ and $\hat{\varphi}_h \in \mathcal{P}_k(T)$ such that $\xib_h = \hat{\xib}_h + \grad \hat{\varphi}_h$ and
    \begin{equation}\label{eq:lclregdecompstab}
        \|\hat{\xib}_h\| + \|\grad \hat{\varphi}_h\| \leq C\|\xib_h\|, \quad \|\grad \hat{\xib}_h\| \leq C\|\curl\xib_h\| 
    \end{equation}
    where $C$ is a positive constant independent of $\xib_h$.
\end{lemma}
\begin{proof}
    Let $\mathsf{R}$ be the \emph{smoothed \Poincare lifting} devised in~\cite[][Definition~3.1 with $l=2$]{costabel_2009} which maps a $2$-form to a $1$-form and serves as a right-inverse of $\curl$. That is, $\curl \mathsf{R}\,\ub = \ub$ for any $\ub \in \Hb(\div0,T)$. In particular, $\mathsf{R}$ can be extended to a bounded operator from $\Hb^s(T)$ to $\Hb^{s+1}(T)$ for any $s \in \R$~\cite[][Corollary~3.4]{costabel_2009}. Importantly, $\mathsf{R}$ maps $\boldsymbol{\mathcal{P}}_{k-1}(T)$ to $\Sigmab_k(T)$ by its nature~\cite[][Section~3]{hiptmair_2009}. Consider the following decomposition
    \begin{equation}
        \xib_h = \underbrace{\mathsf{R}\,\curl\xib_h}_{\hat{\xib}_h} + \underbrace{(\xib_h - \mathsf{R}\,\curl\xib_h)}_{\grad \hat{\varphi}_h}
    \end{equation}
    where $\hat{\xib}_h \in \Sigmab_k(T)$ and $\hat{\varphi}_h \in \mathcal{P}_k(T)$. Note that $\curl \xib_h \in \boldsymbol{\mathcal{P}}_{k-1}(T)$ and thus $\hat{\xib}_h \in \Sigmab_k(T)$. Existence of such $\hat{\varphi}_h \in \mathcal{P}_k(T)$ is justified due to the fact that $\curl (\xib_h - \hat{\xib}_h) = \zerob$ and $\xib_h - \hat{\xib}_h \in \Sigmab_k(T)$. By the boundedness of $\mathsf{R}$~\cite[][Corollary~3.4]{costabel_2009}, we have
    \begin{equation}
        \|\hat{\xib}_h\| \leq C\|\curl \xib_h\|_{\Hb^{-1}(T)} \leq C\|\xib_h\|, \quad \|\grad \hat{\xib}_h\| \leq C\|\curl \xib_h\|.
    \end{equation}
    The estimate of $\|\grad\hat{\varphi}_h\|$ follows from the triangle inequality.
\end{proof}

\subsubsection{A curl-elliptic projection}
Denote two subspaces of $\Hcurl$:
\begin{equation}\label{def:XXh}
    \Xb := \Hcurl \cap \Hcurlz^\perp,\quad\Xb_h := \Sigmabh\cap (\Sigmabh\cap\Hcurlz)^\perp.
\end{equation}
In view of the decompositions~\eqref{def:hodgedecom1}\eqref{def:hodgedecomdisc1}, indeed $\Xb = \Hcurl \cap \curl\Hzcurl$ and $\Xb_h = \curlh \Vb_h$.
Introduce the curl-elliptic projection $\PND:\Hcurl \rightarrow \Xb_h$ by
\begin{equation}\label{def:proj}
    \begin{aligned}
        (\curl(\PND \vb - \vb), \curl\psib_h)  &= 0&&\forall\,\psib_h\in \Sigmabh.
    \end{aligned}
\end{equation}
The following lemma establishes the approximation properties of $\PND$.
\begin{lemma}[approximation properties of $\PND$]\label{thm:curlproj}
    Consider any $\vb \in \Hcurl$. There exists a positive constant $C$ independent of $\vb$ such that the following holds true whenever the right-hand sides are well-defined.
    
    \noindent \emph{($i$)} It holds that 
    \begin{equation}\label{eq:curlprojerrcurl}
        \|\curl(\vb - \PND\vb)\|\leq Ch^{r}|\curl\vb|_{\Hb^{r}}
    \end{equation}
    where $r\in \R, 0 \leq r \leq k$.
    
    \noindent \emph{($ii$)} It holds that
    \begin{equation}\label{eq:superconv}
        (\curl (\vb - \PND\vb), \wb_h) \leq Ch^{r+1/2}|\curl\vb|_{\Hb^{r}}\|\wb_h\|_{\Hdiv}\quad\forall\,\wb_h\in \Vb_{h}
    \end{equation}
    where $r\in \R, 0 \leq r \leq k$.

    \noindent \emph{($iii$)} Assume $\vb \in \Xb$ in addition. It holds that
    \begin{equation}\label{eq:curlprojerrl2}
        \|\vb - \PND\vb\| \leq Ch^{r}\|\vb\|_{\Hb^{r+1/2}}
    \end{equation}
    where $r\in \R, 1/2 \leq r \leq k$.
\end{lemma}
\begin{proof}
    \noindent ($i$) By~\eqref{def:proj}, we have $\|\curl(\vb - \PND\vb)\|^2 = (\curl (\vb - \PND\vb),\curl (\vb - \INDav\vb))$. The approximation property of $\INDav$ (see Remark~\ref{rmk:quasiint}) leads to~\eqref{eq:curlprojerrcurl}.

    \noindent ($ii$) The proof of~\eqref{eq:superconv} closely follows that of~\cite[][Eq.~3.15]{arnold_2012}. We apply the discrete Hodge decomposition~\eqref{def:hodgedecomdisc2} to $\wb_h \in \Vb_{h}$ and obtain $\wb_h = \curl \psib_h + \hb_{h} + \gradh\phi_h$ for some $\psib_h \in \Sigmabh, \hb_{h}\in \hform^2_h, \phi_h \in S_h$. In particular, it is not difficult to see that $(\gradh\phi_h,\phi_h)\in \Vb_h\times S_h$ is the mixed finite element solution of $-\Delta \phi = -\div\wb_h$ with homogeneous Dirichlet conditions. By Theorem~\ref{thm:reglap} we have $\|\phi\|_{H^{s+1}} \leq C\|\div\wb_h\|$ for some $s \in (1/2,1]$. Standard error analysis (see, e.g.,~\cite[][Corollary~12.5.18]{brenner_2007}) shows
    \begin{equation}\label{eq:thmcurlporjtmp-0}
        \|\grad\phi - \gradh\phi_h\|\leq Ch^{s}\|\phi\|_{H^{s+1}}\leq Ch^{s}\|\div\wb_h\|.
    \end{equation}
    Therefore, we insert the decomposition and apply the orthogonality~\eqref{def:proj}:
    \begin{equation}\label{eq:thmcurlporjtmp-1}
        \begin{aligned}
            (\curl (\vb - \PND\vb), \wb_h) &= (\curl (\vb - \PND\vb), \curl\psib_h + \hb_h + \gradh\phi_h) \\
            &= (\curl (\vb - \PND\vb), \hb_h + \gradh\phi_h) \\
            &= (\curl (\vb - \PND\vb), \hb_h - \Pi^{\hform^2}\hb_h + \gradh\phi_h - \grad\phi).
        \end{aligned}
    \end{equation}
    where $\Pi^{\hform^2}:\hform^2_h\rightarrow \hform^2$ denotes an $\Lb^2$-projection. The last step in~\eqref{eq:thmcurlporjtmp-1} is due to the Hodge decomposition~\eqref{def:hodgedecom2}. Appealing to~\cite[][Theorem~5.2]{arnold_2018} that shows the close gap between harmonic forms, we have
    \begin{equation}\label{eq:thmcurlporjtmp-2}
        \|\hb_h - \Pi^{\hform^2}\hb_h\| \leq \|(I - \IRTav)\Pi^{\hform^2}\hb_h\| \leq Ch^s\|\Pi^{\hform^2}\hb_h\|_{\Hb^s} \leq Ch^s\|\Pi^{\hform^2}\hb_h\| \leq Ch^s\|\hb_h\| \leq Ch^s\|\wb_h\|
    \end{equation}
    where we have also applied the approximation property of $\IRTav$ and the embedding in Theorem~\ref{thm:embdd}. 
    Combining~\eqref{eq:thmcurlporjtmp-1}\eqref{eq:thmcurlporjtmp-0}\eqref{eq:thmcurlporjtmp-2} leads to
    \begin{equation}
        (\curl (\vb - \PND\vb), \wb_h)\leq Ch^{s}\|\curl (\vb - \PND\vb)\|\|\wb_h\|_{\Hdiv}.
    \end{equation}
    Insert~\eqref{eq:curlprojerrcurl}, and we obtain~\eqref{eq:superconv}.
    
    \noindent ($iii$) 
    To bound $\|\vb - \PND\vb\|$, we first introduce the $\Lb^2$-projection $\PX:\Xb_h\rightarrow \Xb$ defined by $(\PX \ub_h - \ub_h, \xib) = 0\;\forall\,\xib\in \Xb$ for a given $\ub_h \in \Xb_h$. 
    Immediately, we have
    \begin{equation}\label{eq:qxcurlinvar}
        \curl \PX\ub_h = \curl \ub_h.
    \end{equation}
    It has been shown in~\cite[][Lemma~5.10]{arnold_2006} that
    \begin{equation}\label{eq:thm:curlproj-tmp-1}
        \|\ub_h - \PX \ub_h\| \leq \|\PX \ub_h - \INDav\PX \ub_h \|.
    \end{equation}
    Noting $\Xb \subset \Hcurl\cap\Hzdivz \hookrightarrow \Hb^s$ (see Theorem~\ref{thm:embdd}), we deduce by~\eqref{eq:thm:curlproj-tmp-1}\eqref{eq:qxcurlinvar} and the approximation of $\INDav$ (see Remark~\ref{rmk:quasiint}) that
    \begin{align}
        \|\ub_h - \PX \ub_h\| \leq Ch^{s}\|\PX \ub_h\|_{\Hb^s} \leq Ch^s \|\curl\PX\ub_h\| = Ch^s \|\curl \ub_h\|,\label{eq:qxl2err}
    \end{align}
    where $s\in (1/2,1]$ depends on the geometry of $\Omega$.
    
    Next, we define $\IX :\Hcurl \rightarrow \Xb_h,\,\IX:= \PND\circ\INDav$. By the decomposition~\eqref{def:hodgedecomdisc1}, it is not difficult to see that $\IX$ can be characterized as follows: 
    \begin{equation}\label{eq:charix}
        \IX \ub = \INDav\ub - \grad p_{h,\ub} - \hb_{h,\ub}\quad\forall\,\ub \in \Hcurl
    \end{equation}
    where $p_{h,\ub} \in P_h$ and $\hb_{h,\ub} \in \hform^1_h$ satisfy    
    \begin{equation}\label{def:ph}
         (\grad p_{h,\ub} + \hb_{h,\ub}, \grad q_h + \jb_h) = (\INDav\ub,\grad q_h+ \jb_h)\quad\forall\, q_h\in P_h, \jb_h \in \hform^1_h.
    \end{equation}
    Since $\vb \in \Xb$, it holds that
    \begin{equation}
        (\grad p_{h,\vb} + \hb_{h,\vb}, \grad q_h + \jb_{h}) = (\INDav\vb - \vb,\grad q_h + \jb_{h})\quad\forall\, q_h\in P_h, \jb_h \in \hform^1_h.
    \end{equation}
    which implies that
    \begin{equation}
        \|\grad p_{h,\vb} + \hb_{h,\vb}\| \leq C \|\vb - \INDav\vb\| \leq Ch^{r_1}|\vb|_{\Hb^{r_1}}
    \end{equation}
    where $0 \leq r_1 \leq k$.
    Consequently, we have the following interpolation error estimates for $\vb \in \Xb$:
    \begin{equation}\label{eq:intastl2errest}
        \|\vb - \IX\vb\| \leq \|\vb - \INDav\vb\| + \|\grad p_{h,\vb} + \hb_{h,\vb}\| \leq C h^{r_1} |\vb|_{\Hb^{r_1}},
    \end{equation}
    \begin{equation}\label{eq:intastcurlerrest}
        \|\curl(\vb - \IX\vb)\| = \|\curl(\vb - \INDav\vb)\| \leq C h^{r_1} |\curl\vb|_{\Hb^{r_1}}.
    \end{equation}
    Combining the properties of $\PX$ and $\IX$, we claim that
    \begin{equation}\label{eq:l2errestqxiast}
        \|\IX \vb - \PX \IX \vb\| \leq Ch^{r_1} |\vb|_{\Hb^{r_1}}.
    \end{equation}
    Indeed, it holds that
    \begin{equation}
        \begin{aligned}
             \|\IX \vb - \PX \IX \vb\|^2 &= (\IX \vb - \PX \IX\vb, \IX \vb - \PX \IX \vb)\\
            &= (\IX \vb - \PX \IX\vb, \IX \vb - \vb)\\
            &\leq \|\IX \vb - \PX \IX\vb\|\|\IX \vb - \vb\|.
        \end{aligned}
    \end{equation}
    Insert~\eqref{eq:intastl2errest}, and we obtain~\eqref{eq:l2errestqxiast}.

    To estimate $\|\vb - \PND\vb\|$, we use the fact that $\PND \IX = \IX$ to deduce that
    \begin{equation}\label{eq:curlprojdual-4}
        \begin{aligned}
            \|\vb - \PND\vb\| &\leq \|\vb - \PX \PND\vb\| + \|\PX \PND\vb - \PND\vb\| \\
            &\leq \|\vb - \PX \PND\vb\| + \|\PX \PND(\vb-\IX\vb) - \PND(\vb - \IX \vb)\| + \|\PND \IX \vb - \PX\PND \IX\vb\| \\
            &= \underbrace{\|\vb - \PX \PND\vb\|}_{I_1} + \underbrace{\|\PX \PND(\vb-\IX\vb) - \PND(\vb - \IX \vb)\|}_{I_2} + \underbrace{\|\IX \vb - \PX \IX\vb\|}_{I_3}.
        \end{aligned}
    \end{equation}
    Using~\eqref{eq:qxl2err}\eqref{eq:curlprojerrcurl}\eqref{eq:intastcurlerrest}, the term $I_2$ is bounded as
    \begin{equation}
        I_2 \leq Ch^s\|\curl\PND(\vb-\IX\vb)\|\leq Ch^s\|\curl(\vb-\IX\vb)\|\leq Ch^{s+r_2-1/2}|\curl \vb|_{\Hb^{r_2-1/2}} \leq Ch^{r_2}|\vb|_{\Hb^{r_2 + 1/2}}
    \end{equation}
    where $1/2 \leq r_2 \leq 1/2+k$.
    The term $I_3$ can be bound as in~\eqref{eq:l2errestqxiast}. To bound $I_1$, we introduce the dual problem seeking $\wb \in \Xb$ such that
    \begin{equation}\label{eq:dualcurlproj}
        (\curl \wb,\curl\psib) = (\vb - \PX \PND\vb, \psib)\quad\forall\,\psib\in \Xb.
    \end{equation}
    The dual solution $\wb \in \Xb$ is well-defined since the bilinear form is coercive in $\Xb$.
    Since~\eqref{eq:dualcurlproj} implies that $\curl\curl\wb = \vb - \PX \PND\vb \in L^2$ and $\curl\wb\times\nb = 0$ on $\partial\Omega$, we have $\curl\wb \in \Hzcurl\cap\Hdiv \hookrightarrow \Hb^{s}$ for some $s > 1/2$ by Theorem~\ref{thm:embdd} and
    \begin{equation}\label{eq:curlprojdualreg}
        \|\curl\wb\|_{\Hb^{s}} \leq C\|\vb - \PX \PND\vb\|.
    \end{equation}
    Let $\psib = \vb - \PX \PND\vb$ in~\eqref{eq:dualcurlproj}. We obtain by~\eqref{eq:qxcurlinvar}\eqref{def:proj}\eqref{eq:canonicalapp}\eqref{eq:curlprojdualreg} that
    \begin{equation}\label{eq:curlprojdual-1}
        \begin{aligned}
            \|\vb - \PX \PND\vb\|^2 &= (\curl \wb,\curl (\vb - \PX \PND\vb)) = (\curl \wb,\curl (\vb - \PND\vb))\\
            &=(\curl (\wb - \IND\wb), \curl (\vb - \PND\vb)) \\
            &\leq Ch^s|\curl\wb|_{\Hb^s}\|\curl(\vb - \PND\vb)\| \\
            &\leq Ch^s\|\vb - \PX \PND\vb\|\|\curl(\vb - \PND\vb)\|.
        \end{aligned}
    \end{equation} 
    Inserting~\eqref{eq:curlprojerrcurl}, we obtain $I_1 \leq Ch^{s+r_3-1/2}|\curl \vb|_{\Hb^{r_3-1/2}} \leq Ch^{r_3}|\vb|_{\Hb^{r_3+1/2}}$ where $1/2 \leq r_3 \leq 1/2 + k$. Combining the bounds of $I_1, I_2, I_3$ leads to~\eqref{eq:curlprojerrl2}.
\end{proof}

\subsection{Norm equivalence on \texorpdfstring{$\Sigmab_h$}{Sigmah}}\label{sec:stabdisc}
In line with the analysis of the continuous case, we endow $\Sigmabh$ with the nonstandard norm \[\|\taub_h\|_{\Sigmabh}^2 := \|\taub_h\|^2 + \|\curl\taub_h\|^2_{\Vb_{h,0}'}.\] Note that $\|\taub\|_{\Sigmabh}\leq \|\taub\|_{\Sigmab}$ for $\taub\in \Sigmabh+\Sigmab$. In this section, we show an equivalent norm on $\Sigmab_h$ as the discrete counterpart of Lemma~\ref{thm:sigmanormequiv}. 

Given $\taub_h \in \Sigmabh$, we define $\Pzh:\Sigmabh \rightarrow \Sigmab_{h,0}\cap(\Sigmab_{h,0}\cap \Hzcurlz)^\perp$ such that $\taub_{h0} := \Pzh\taub_h$ solves
\begin{equation}\label{def:tau0h}
    (\curl\taub_{h0},\curl\psib_h) = (\curl\taub_h,\curl\psib_h)\quad\forall\,\psib_h\in\Sigmab_{h,0}.
\end{equation}
Since vanishing $\curl\taub_h$ implies vanishing $\taub_{h0}$, the equation~\eqref{def:tau0h} admits a uniquely defined solution $\taub_{h0}$. 
\begin{lemma}[equivalent norm on $\Sigmab_h$]\label{thm:sigmanormequivdisc} 
    For any $\taub_h \in \Sigmabh$, it holds that
    \begin{equation}\label{eq:sigmahnormeq}
         \|\taub_h\|_{\Sigmabh} \simeq \|\taub_h\| + \|\Pzh\taub_h\|_{\Hcurl}
    \end{equation}
    where $\Pzh:\Sigmabh \rightarrow \Sigmab_{h,0}\cap(\Sigmab_{h,0}\cap \Hzcurlz)^\perp$ is defined in~\eqref{def:tau0h}.
\end{lemma}
\begin{proof}
    Applying the regular decomposition of $\Hzdiv$ in Theorem~\ref{thm:regdecomphdiv} and the properties of $\IRTav, \INDav$ (see Remark~\ref{rmk:quasiint}), we deduce that
    \begin{equation}
        \begin{aligned}
            \|\curl\taub_h\|_{\Vb_{h,0}'} 
            &= \sup_{\zerob \neq \vb_h \in \Vb_{h,0}} \frac{( \curl \taub_h, \vb_h)}{\|\vb_h\|_{\Hzdiv} } = \sup_{\zerob \neq \vb_h \in \Vb_{h,0}} \frac{(\curl \taub_h, \IRTav(\hat{\vb} + \curl \hat{\taub}))}{\|\vb_h\|_{\Hzdiv} } 
            \\
            &= \sup_{\zerob \neq \vb_h \in \Vb_{h,0}} \frac{( \curl \taub_h, \IRTav \hat{\vb} + \curl\INDav\hat{\taub})}{\|\vb_h\|_{\Hzdiv} } \\
            &= \sup_{\zerob \neq \vb_h \in \Vb_{h,0}} \frac{( \curl \taub_h, \hat{\vb}) + ( \curl \taub_h, \IRTav\hat{\vb} - \hat{\vb}) + (\curl\taub_h, \curl\INDav\hat{\taub})}{\|\vb_h\|_{\Hzdiv} } \\
            &= \sup_{\zerob \neq \vb_h \in \Vb_{h,0}} \frac{(\taub_h, \curl\hat{\vb}) + (\curl \taub_h, \IRTav\hat{\vb}- \hat{\vb}) + (\curl\Pzh\taub_h, \curl\INDav\hat{\taub})}{\|\vb_h\|_{\Hzdiv} } \\
            &\leq \sup_{\zerob \neq \vb_h \in \Vb_{h,0}} \frac{\|\taub_h\| \|\hat{\vb}\|_{\Hb^1}+Ch\|\curl\taub_h\||\hat{\vb}|_{\Hb^1} + \|\curl\Pzh\taub_h\|\|\curl \hat{\taub}\|}{\|\vb_h\|_{\Hzdiv} }\\
            &\leq C(\|\taub_h\| + \|\curl\Pzh \taub_h\|)
        \end{aligned}
    \end{equation}
    where we applied the inverse inequality and~\eqref{eq:regdecomphdiv} in the last step. It is clear from~\eqref{def:tau0h} that $\|\curl\Pzh \taub_h\|\leq \|\curl \taub_h\|_{\Vb_{h,0}'}$, and we can conclude.
\end{proof}
The following lemma implies the well-posedness and stability of the discrete problem~\eqref{eq:hodgelapdisc}. 
\begin{lemma}[boundedness and inf-sup condition]\label{thm:bddinfsupdisc}
     There exist positive constants $C, \beta_\a'$ independent of $h$ such that
    \begin{align}
        \a((\mub_h,\ub_h),(\taub_h,\vb_h)) \leq C(\|(\mub_h,\ub_h)\|_{\Sigmabh\times\Hdiv}\|(\taub_h,\vb_h)\|_{\Sigmabh\times\Hdiv} &  && \forall\,(\mub_h,\ub_h), (\taub_h,\vb_h)\in \Sigmabh\times \Vb_{h,0},\label{eq:bddadisc}\\
        \sup_{0\neq (\taub_h, \vb_h)\in\Sigmabh\times \Vb_{h,0}}\frac{\a((\mub_h,\ub_h),(\taub_h,\vb_h))}{\|(\taub_h,\vb_h)\|_{\Sigmabh\times\Hdiv} } \geq \beta_\a'\|(\mub_h,\ub_h)\|_{\Sigmabh\times\Hdiv} &  && \forall\,(\mub_h,\ub_h) \in \Sigmabh\times\Vb_{h,0}.\label{eq:infsupadisc}
    \end{align}
\end{lemma}
\begin{proof}
    The proof is completely parallel to that of Lemma~\ref{thm:bddinfsup} by leveraging the discrete Helmholtz decomposition~\eqref{eq:helmdecompdisc2}, the $\Sigmabh$-norm equivalence~\eqref{eq:sigmahnormeq} and the discrete \Poincare inequality~\cite[][Theorem~4.6]{arnold_2018}.
\end{proof}
\begin{remark}[$\Sigmab$- and $\Sigmabh$-norms]
 Note that Lemmas~\ref{thm:bddinfsup} and~\ref{thm:bddinfsupdisc} state the boundedness and inf-sup conditions in terms of different norms. In particular, $\|\cdot\|_{\Sigmabh}$ is weaker than $\|\cdot\|_{\Sigmab}$. This fact can be understood as a consequence of the violation of the de Rham complex structure, which causes $h^{1/2}$-suboptimal convergence.
\end{remark}

\subsection{A priori error analysis}\label{sec:errest}
We establish \emph{a priori} error estimates with respect to energy norms in the following theorem.
\begin{theorem}[error estimates in the energy norm]\label{thm:eneerrest}
    Let $\ub\in \Hb^1_0$ solve~\eqref{eq:veclap} and $(\mub_h, \ub_h)\in \Sigmabh\times \Vb_{h,0}$ solve~\eqref{eq:hodgelapdisc}. If $\ub\in \Wb^{k,\infty}\cap \Hb^{k+1}$, there exists a positive constant $C$ such that 
    \begin{equation}\label{eq:eneerrest}
        \|\ub - \ub_h\|_{\Hdiv} + \|\mub - \mub_h\|_{\Sigmabh} + h\|\curl(\mub - \mub_h)\| \leq C h^{k - 1/2}\left(|\ub|_{\Wb^{k,\infty}} + \|\ub\|_{\Hb^{k+1}}\right)
    \end{equation}
    where $\mub := \curl\ub$.
\end{theorem}
\begin{proof}
    Denote by $\ub_I := \IRT\ub \in \Vb_{h,0}$ the canonical Raviart-Thomas interpolant and $\mub_\Pi := \PND \mub \in \Xb_h \subset \Sigmabh$ the curl-elliptic projection defined in~\eqref{def:proj}. Thanks to the approximation properties~\eqref{eq:canonicalapp}\eqref{eq:curlprojerrcurl}\eqref{eq:superconv}\eqref{eq:curlprojerrl2} and the inverse inequality, it suffices to estimate $\|\ub_I - \ub_h\|_{\Hdiv} + \|\mub_\Pi - \mub_h\|_{\Sigmabh}$.
    
    By the inf-sup condition~\eqref{eq:infsupadisc} and the Galerkin orthogonality from~\eqref{eq:hodgelaprd}\eqref{eq:hodgelapdiscrd}, we have
    \begin{equation}
        \begin{aligned}
            \|\ub_I - \ub_h\|_{\Hdiv} + \|\mub_\Pi - \mub_h\|_{\Sigmabh} &\leq \beta_\a^{'-1}\sup_{0\neq (\taub_h, \vb_h)\in\Sigmabh\times \Vb_{h,0}}\frac{\a((\mub_\Pi - \mub_h,\ub_I - \ub_h),(\taub_h,\vb_h))}{\|\taub_h\|_{\Sigmabh} + \|\vb_h\|_{\Hdiv}} \\
            &=\beta_\a^{'-1}\sup_{0\neq (\taub_h, \vb_h)\in\Sigmabh\times \Vb_{h,0}}\frac{\a((\mub_\Pi - \mub,\ub_I - \ub),(\taub_h,\vb_h))}{\|\taub_h\|_{\Sigmabh} + \|\vb_h\|_{\Hdiv}}.
        \end{aligned}
    \end{equation}
    The numerator is, by definition,
    \begin{equation}\label{eq:thmerr-1}
    \begin{aligned}
        \a((\mub_\Pi - \mub,\ub_I - \ub),&(\taub_h,\vb_h)) =  -\underbrace{(\curl\taub_h, \ub_I - \ub)}_{I_1}  \\
        &\quad\quad + \underbrace{(\mub_\Pi - \mub,\taub_h) + (\curl(\mub_\Pi - \mub), \vb_h) + (\div(\ub_I - \ub),\div\vb_h)}_{I_2}.
    \end{aligned}
    \end{equation}
    Note that $\mub := \curl\ub \in \Xb$ and thus~\eqref{eq:curlprojerrl2} is applicable. By~\eqref{eq:canonicalapp}\eqref{eq:curlprojerrcurl}\eqref{eq:superconv}\eqref{eq:curlprojerrl2}, $I_2$ is bounded by
    \begin{equation}\label{eq:thmerr-2}
        \begin{aligned}
             I_2 &\leq Ch^{k-1/2}\|\mub\|_{\Hb^{k}}\|\taub_h\| + Ch^{k-1/2}|\mub|_{\Hb^{k}}\|\div\vb_h\| + Ch^k|\ub|_{H^{k+1}}\|\div\vb_h\|  \\
             &\leq Ch^{k-1/2}\|\ub\|_{\Hb^{k+1}}(\|\taub_h\| + \|\vb_h\|_\Hdiv).
        \end{aligned}
    \end{equation}
    
    The main effort is dedicated to estimating $I_1$ in~\eqref{eq:thmerr-1} since $\|\curl\taub_h\|$ is not controlled by $\|\taub_h\|_{\Sigmabh}$. Naive application of the inverse inequality will lead to a severely pessimistic bound.
    Recall the projection $\Pzh$ defined in~\eqref{def:tau0h}. We have
    \begin{equation}\label{eq:thmerr-5}
        I_1 = (\curl(\taub_h -\Pzh\taub_h), \ub_I - \ub) + (\curl \Pzh\taub_h, \ub_I - \ub).
    \end{equation}
    The second term on the right-hand side of~\eqref{eq:thmerr-5} can be easily bounded by
    \begin{equation}
        (\curl \Pzh\taub_h, \ub_I - \ub) \leq Ch^k|\ub|_{\Hb^k}\|\curl\Pzh\taub_h\| \leq Ch^k|\ub|_{\Hb^k}\|\taub_{h}\|_{\Sigmabh}
    \end{equation}
    thanks to the $\Sigmabh$-norm equivalence (see Lemma~\ref{thm:sigmanormequivdisc}). The main difficulty is due to the first term on the right-hand side of~\eqref{eq:thmerr-5}. Let $\pstar \in [1,2]$ be the \Holders conjugate of $p\in [2,\infty]$, that is, $1/\pstar + 1/p = 1$. By \Holders inequality and~\eqref{eq:canonicalapp}, we have 
    \begin{equation}\label{eq:thmerr-3}
        (\curl(\taub_h -\Pzh\taub_h), \ub_I - \ub) \leq \|\curl (\taub_h - \Pzh\taub_h)\|_{\Lb^\pstar}\|\ub_I - \ub\|_{\Lb^p} \leq Ch^k\|\curl (\taub_h - \Pzh\taub_h)\|_{\Lb^\pstar}|\ub|_{\Wb^{k,p}}.
    \end{equation}
    
    Before estimating the $L^\pstar$-term, we introduce a partition of the domain $\Omega$. Denote $A_0 := \{\xb\in\Omega: \dist(\xb,\partial\Omega) \leq ch\}$ where $c$ is a constant sufficiently large but independent of $h$. Also, denote 
    \begin{equation}
        A_j := \{\xb\in\Omega: d_{j-1} < \dist(\xb, \partial\Omega) \leq d_{j} \},\,j = 1,\ldots,J\text{ with }d_j:= ch\cdot2^{j},
    \end{equation}
   where $J \in \mathbb{N}$ is the maximum index such that $A_J \neq\emptyset$. We then have that $\Omega = \bigcup_{j=0}^J A_j$. See Figure~\ref{fig:partition} for a sketch of the partition.
   \begin{figure}
       \centering
       \input{fig_partition.tex}
       \caption{Sketch of the partition $\Omega = \bigcup_{j=0}^J A_j$}
       \label{fig:partition}
   \end{figure}
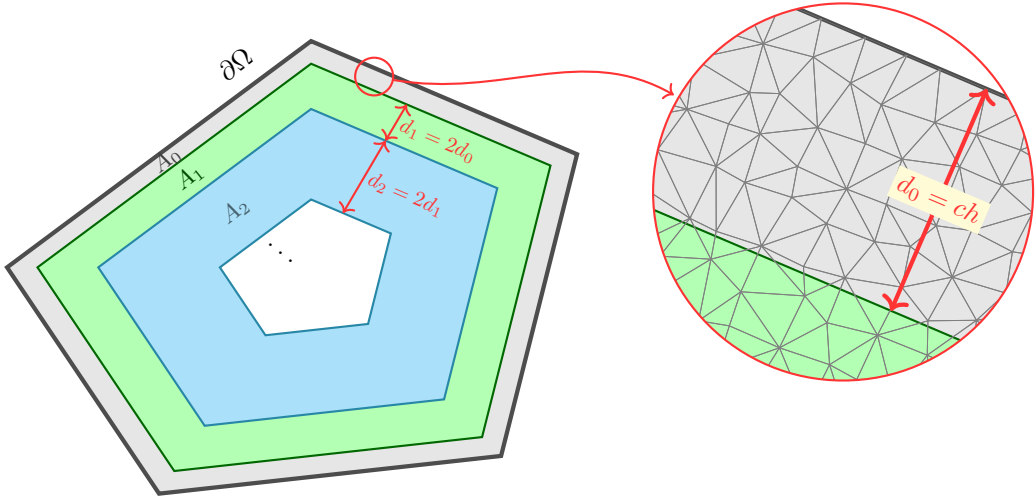
   This partition is motivated by the \emph{dyadic annuli decomposition}, which is an instrumental tool in showing pointwise estimates~\cite{wahlbin_1991}. The enlarged subdomains are denoted by
   \begin{equation}
       A_j' := A_{j-1}\cup A_j\cup A_{j+1},\, j=1,\ldots,J-1, \quad A_0':=A_0\cup A_1,\quad A_J':= A_{J-1}\cup A_J.\\
   \end{equation}
   It is easy to see that 
   \begin{equation}
       |A_j|\leq Cd_j,\quad\dist(\partial A_j, \partial A'_j) \simeq d_j, j = 1,\ldots,J,
   \end{equation}
   and $J\sim\log h$, assuming $\text{diam}(\Omega)\sim1$. 
    
    Furthermore, we need the following crucial lemma, which will be proved later.
    \begin{lemma}[local estimate of discrete curl-harmonic functions]\label{thm:lclest}
        Let $\Omega'\subset\subset\Omega''\subset\subset\Omega$ and $d := \dist(\partial\Omega',\partial\Omega'') > ch$ where $c$ is a constant sufficiently large and independent of $h$.
        If $\xib_h \in \Sigmabh$ satisfies 
        \begin{equation}\label{def:curlharm}
            (\curl\xib_h,\curl\psib_h) = 0 \quad\forall\,\psib_h\in\Sigmab_{h,0},
        \end{equation}
        then there exists a positive constant $C$ independent of $h$ such that
        \begin{equation}\label{eq:lclest}
            \|\curl\xib_h\|_{\Lb^2(\Omega')} \leq Cd^{-1}\|\xib_h\|_{\Lb^2(\Omega'')}.
        \end{equation}
    \end{lemma}
    \noindent In view of~\eqref{def:tau0h}, $\taub_h - \Pzh\taub_h$ satisfies~\eqref{def:curlharm}. Since $\Omega = \bigcup_{j=0}^J A_j$, we obtain that
    \begin{equation}\label{eq:thmerr-4}
        \begin{aligned}
            \|\curl(\taub_h - \Pzh\taub_h)\|_{\Lb^\pstar}^\pstar &= \|\curl(\taub_h - \Pzh\taub_h)\|^\pstar_{\Lb^\pstar(A_0)} + \sum_{j=1}^J\|\curl(\taub_h - \Pzh\taub_h)\|^\pstar_{\Lb^\pstar(A_j)} \\
            &\leq C|A_0|^{1-\frac{\pstar}{2}}\|\curl(\taub_h - \Pzh\taub_h)\|^\pstar_{\Lb^2(A_0)} + \sum_{j=1}^J C|A_j|^{1-\frac{\pstar}{2}}\|\curl(\taub_h - \Pzh\taub_h)\|^\pstar_{\Lb^2(A_j)} \\
            &\leq Ch^{1-\frac{\pstar}{2}}\|\curl(\taub_h - \Pzh\taub_h)\|^\pstar_{\Lb^2(A_0)} + \sum_{j=1}^J Cd_j^{1-\frac{\pstar}{2}}\|\curl(\taub_h - \Pzh\taub_h)\|^\pstar_{\Lb^2(A_j)} \\
            &\leq Ch^{1-\frac{3\pstar}{2}}\|\taub_h - \Pzh\taub_h\|^q_{\Lb^2(A_0')} + \sum_{j=1}^J Cd_j^{1-\frac{3\pstar}{2}}\|\taub_h - \Pzh\taub_h\|^\pstar_{\Lb^2(A_j')}
        \end{aligned}
    \end{equation}
    where we applied Lemma~\ref{thm:lclest} in the last step.
    Recall the \Holders inequality $\sum_{k=1}^n|x_ky_k|\leq\left(\sum_{k=1}^n|x_k|^r\right)^{1/r}\left(\sum_{k=1}^n|y_k|^t\right)^{1/t}$ for $r,t\in(1,\infty),1/r+1/t=1$. Hence, for $q \in [1,2)$, we obtain
    \begin{equation}\label{eq:thmerr-4-2}
        \begin{aligned}
            \|\curl(\taub_h - \Pzh\taub_h)\|_{\Lb^\pstar}^\pstar 
            &\leq Ch^{1-\frac{3\pstar}{2}}\|\taub_h - \Pzh\taub_h\|^\pstar + C \left(\sum_{j=1}^J d_j^{\frac{2 - 3\pstar}{2-\pstar}}\right)^{\frac{2-q}{2}}\left(\sum_{j=1}^J \|\taub_h - \Pzh\taub_h\|^2_{\Lb^2(A_j')}\right)^{\pstar/2} \\
            &\leq Ch^{1-\frac{3\pstar}{2}}\|\taub_h - \Pzh\taub_h\|^q \\
            &\leq Ch^{1-\frac{3\pstar}{2}}\|\taub_h\|_{\Sigmabh}^q 
        \end{aligned}
    \end{equation}
    Here, we conducted the following elementary calculation
    \begin{equation}
        \sum_{j=1}^J d_j^{\frac{2 - 3\pstar}{2-\pstar}} = \sum_{j=1}^J (ch\cdot2^j)^{\frac{2 - 3\pstar}{2-\pstar}} =(ch)^{\frac{2 - 3\pstar}{2-\pstar}} \sum_{j=1}^J 2^{\frac{2 - 3\pstar}{2-\pstar}j} < (ch)^{\frac{2 - 3\pstar}{2-\pstar}} \left(2^{\frac{2 - 3\pstar}{\pstar-2}} - 1\right)^{-1} 
    \end{equation} and used the $\Sigmabh$-norm equivalence~\eqref{eq:sigmahnormeq}. The conclusion follows from setting $p = \infty, \pstar = 1$, the estimates~\eqref{eq:thmerr-1}\eqref{eq:thmerr-2}\eqref{eq:thmerr-3}\eqref{eq:thmerr-4} and the triangle inequality.
\end{proof}

\begin{proof}[Proof of Lemma~\ref{thm:lclest}]
    Introduce an intermediate subdomain $\Omega_h\subset\Omega$ that is the union of elements and satisfies $\Omega'\subset\subset\Omega_h\subset\subset\Omega''$ and $\dist(\partial\Omega',\partial\Omega_h)>d/2$. Note that such $\Omega_h$ can always be found since we assume $\dist(\partial\Omega',\partial\Omega'') > ch$ with a sufficiently large $c$.
    Let $w \in C^\infty_c(\Omega)$ be a cut-off scalar function such that $w \equiv 1$ on $\Omega'$ and $\supp(w)\subset\Omega_h$. We require that $|\partial^{r} w(\xb)| \leq Cd^{-r}$ for any $\xb \in \Omega$ and $r \in \mathbb{N}$ with $r \geq 1$. 
    To show~\eqref{eq:lclest}, we start with
    \begin{equation}\label{eq:lemlclest-1}
        \begin{aligned}
            \|w\curl\xib_h\|^2 &= (\curl\xib_h,w^2\curl\xib_h) \\
            &= (\curl\xib_h,\curl(w^2\xib_h)) - (\curl\xib_h,2w\grad w\times\xib_h)\\
            &= \underbrace{(\curl\xib_h,\curl(w^2\xib_h - \IND (w^2\xib_h)))}_{I_1} - \underbrace{(2w\curl\xib_h,\grad w\times\xib_h)}_{I_2}
        \end{aligned}
    \end{equation}
    where we have used the curl-harmonic property~\eqref{def:curlharm}.
    By~\eqref{eq:canonicalapp}, $I_1$ is bounded by
    \begin{align}
        I_1 &\leq Ch^{k}\|\curl\xib_h\|_{\Lb^2(\Omega_h)}\left(\sum_{T:T\subset \Omega_h}|\curl(w^2\xib_h)|^2_{\Hb^{k}(T)}\right)^{1/2}.\label{eq:lemlclest-8}
    \end{align}
    Note that $\partial^k(\curl\xib_h)$ vanishes. We bound each term in the summation by
    \begin{equation}\label{eq:lemlclest-9}
    \begin{aligned}
        |\curl(w^2\xib_h)|_{\Hb^{k}(T)} &\leq |\grad (w^2) \times\xib_h|_{\Hb^{k}(T)} + |w^2\curl\xib_h|_{\Hb^{k}(T)}\\
        &\leq C\sum_{r = 0}^{k} \underbrace{\|\partial^{r+1}(w^2)\|_{L^\infty(T)}\|\partial^{k-r}\xib_h\|_{\Lb^2(T)}}_{I_{1,1}} \\
        &\quad\quad + C\sum_{r=1}^{k} \underbrace{\|\partial^{r}(w^2)\|_{L^\infty(T)}\|\partial^{k-r}(\curl\xib_h)\|_{\Lb^2(T)}}_{I_{1,2}}.
    \end{aligned}
    \end{equation}
    $I_{1,2}$ is directly addressed by the inverse inequality and $d>ch$:
    \begin{equation}\label{eq:lemlclest-14}
        I_{1,2} \leq Cd^{-r}h^{-k+r}\|\curl\xib_h\|_{\Lb^2(T)} \leq Cd^{-1}h^{-k+1}\|\curl\xib_h\|_{\Lb^2(T)} .
    \end{equation}
    To bound $I_{1,1}$, we decompose $\xib_h = \hat{\xib}_h + \grad \hat{\varphi}_h$ where $\hat{\xib}_h \in \Sigmab_k(T), \hat{\varphi}_h \in \mathcal{P}_k(T)$ according to the local regular decomposition in Lemma~\ref{thm:lclregdecomp}. The following three cases are addressed separately: 
    \begin{itemize}
        \item when $r=0$,
        \begin{equation}\label{eq:lemlclest-12}
            \begin{aligned}
                I_{1,1} &= \|\partial(w^2)\|_{L^\infty(T)}\|\partial^{k}(\hat{\xib}_h+\grad\hat{\varphi}_h)\|_{\Lb^2(T)} \\
                &\leq Cd^{-1}\left(\|\partial^{k-1}\grad \hat{\xib}_h\|_{\Lb^2(T)} + \|\partial^{k}\grad\hat{\varphi}_h\|_{\Lb^2(T)}\right) \\
                &=Cd^{-1}\|\partial^{k-1}\grad \hat{\xib}_h\|_{\Lb^2(T)} \\
                &\leq Cd^{-1}h^{-k+1}\|\grad \hat{\xib}_h\|_{\Lb^2(T)} \\
                &\leq Cd^{-1}h^{-k+1}\|\curl\xib_h\|_{\Lb^2(T)}
            \end{aligned}
        \end{equation}
        where we have used the property of $w$, the stability estimate~\eqref{eq:lclregdecompstab}, and the fact that $\partial^{k}\grad\hat{\varphi}_h$ vanishes;
        \item when $1\leq r\leq k-1$,
        \begin{equation}\label{eq:lemlclest-13}
            \begin{aligned}
                I_{1,1} &= \|\partial^{r+1}(w^2)\|_{L^\infty(T)}\|\partial^{k-r}(\hat{\xib}_h+\grad\hat{\varphi}_h)\|_{\Lb^2(T)} \\
                &\leq Cd^{-r-1}\left(\|\partial^{k-r-1}\grad \hat{\xib}_h\|_{\Lb^2(T)} + \|\partial^{k-r}\grad\hat{\varphi}_h\|_{\Lb^2(T)}\right) \\
                &\leq Cd^{-r-1}\left(h^{-k+r+1}\|\grad \hat{\xib}_h\|_{\Lb^2(T)} + h
                ^{-k+r}\|\grad\hat{\varphi}_h\|_{\Lb^2(T)}\right) \\
                &\leq Cd^{-r-1}\left(h^{-k+r+1}\|\curl\xib_h\|_{\Lb^2(T)} + h
                ^{-k+r}\|\xib_h\|_{\Lb^2(T)}\right) \\
                &\leq Cd^{-1}h^{-k+1}\left(\|\curl\xib_h\|_{\Lb^2(T)} + d^{-1}\|\xib_h\|_{\Lb^2(T)}\right)
            \end{aligned}
        \end{equation}
        where we have used the property of $w$, the stability estimates~\eqref{eq:lclregdecompstab}, and $d > ch$;
        \item when $r = k$,
        \begin{equation}\label{eq:lemlclest-15}
            I_{1,1} = \|\partial^{k+1}(w^2)\|_{L^\infty(T)}\|\xib_h\|_{\Lb^2(T)} \leq C d^{-k-1}\|\xib_h\|_{\Lb^2(T)} \leq Ch^{-k+1}d^{-2}\|\xib_h\|_{\Lb^2(T)}
        \end{equation}
        where we have used the property of $w$ and $d > ch$.
    \end{itemize}
    Combining~\eqref{eq:lemlclest-8}\eqref{eq:lemlclest-9}\eqref{eq:lemlclest-12}\eqref{eq:lemlclest-13}\eqref{eq:lemlclest-15}\eqref{eq:lemlclest-14}, we obtain
    \begin{equation}\label{eq:lemlclest-11}
    \begin{aligned}
        I_1 &\leq C\|\curl\xib_h\|_{\Lb^2(\Omega_h)}\left(hd^{-2}\|\xib_h\|_{\Lb^2(\Omega_h)} + hd^{-1}\|\curl\xib_h\| _{\Lb^2(\Omega_h)}\right) \\
        &\leq Cd^{-2}\|\xib_h\|_{\Lb^2(\Omega_h)}^2 + Chd^{-1}\|\curl\xib_h\|_{\Lb^2(\Omega_h)}^2
    \end{aligned}
    \end{equation}
    by the inverse inequality.
    Term $I_2$ is directly bounded by
    \begin{equation}\label{eq:lemlclest-10}
       I_2 \leq Cd^{-1}\|w\curl\xib_h\|\|\xib_h\|_{L^2(\Omega_h)} \leq \frac{1}{2}\|w\curl\xib_h\|^2 + Cd^{-2}\|\xib_h\|_{L^2(\Omega_h)}^2.
    \end{equation}
    The first term in~\eqref{eq:lemlclest-10} can be absorbed into the left-hand side of~\eqref{eq:lemlclest-1}.
    Combining~\eqref{eq:lemlclest-1}\eqref{eq:lemlclest-11}\eqref{eq:lemlclest-10} and noting that $\|\curl\xib_h\|_{L^2(\Omega')} \leq \|w\curl\xib_h\|$, we have 
    \begin{equation}\label{eq:lemlclest-6}
        \|\curl\xib_h\|^2_{L^2(\Omega')} \leq Chd^{-1}\|\curl\xib_h\|_{\Lb^2(\Omega'')}^2 + Cd^{-2}\|\xib_h\|^2_{\Lb^2(\Omega'')}.
    \end{equation}
    Applying~\eqref{eq:lemlclest-6} to term $\|\curl\xib_h\|^2_{L^2(\Omega'')}$ on the right-hand side (without changing $\Omega''$ by abuse of notation) yields 
    \begin{equation}
        \|\curl\xib_h\|^2_{L^2(\Omega')} \leq Ch^2d^{-2}\|\curl\xib_h\|^2_{L^2(\Omega'')} + Cd^{-2}\|\xib_h\|_{L^2(\Omega'')}^2.
    \end{equation}
    The conclusion~\eqref{eq:lclest} follows from applying the inverse inequality once again (the necessary enlargement of the subdomain $\Omega''$ is once again omitted by abuse of notation).
\end{proof}
\begin{remark}[alternative proof of Lemma~\ref{thm:lclest}]
    Estimate~\eqref{eq:lclest} is a discrete Caccioppoli-type inequality. The original Caccioppoli inequality is stated for scalar harmonic functions. Its discrete counterpart is a crucial tool in pointwise and $L^p$ estimates of discretizations of scalar elliptic equations; see, e.g.,~\cite{schatz_1977,wahlbin_1991,chen_2004,leykekhman_2021}. Its counterpart for curl-elliptic problems (that is, Lemma~\ref{thm:lclest}) is rather new and has been established for studying discretizations of Maxwell's equations~\cite{faustmann_2022, ma_2025}. The lowest-order case is first shown in~\cite{faustmann_2022} while recently it has been extended to high-order cases in~\cite{ma_2025} relying on a thorough analysis of \Nedelec elements~\cite[][Lemma~2.4]{ma_2025}. We have provided an alternative proof of the same result relying on the smoothed \Poincare lifting (see Lemma~\ref{thm:lclregdecomp}).
\end{remark}

The $\Lb^2$-error bound for $\ub_h$ is improved using the duality technique. 
\begin{theorem}[error estimate of $\ub_h$ in $\Lb^2$]\label{thm:l2errest}
    Assume that $\Omega$ is convex. There exists a positive constant $C$ such that 
    \begin{equation}\label{eq:l2errest}
        \|\ub - \ub_h\| \leq \begin{cases}
            C h^{k}\left(|\ub|_{\Wb^{k,\infty}} + \|\ub\|_{\Hb^{k+1}}\right) & \text{if } k \geq 2,  \\
            C h^{5/6}\left(|\ub|_{\Wb^{1,\infty}} + \|\ub\|_{\Hb^{2}}\right) & \text{if } k = 1. 
        \end{cases}
    \end{equation}
\end{theorem}
\begin{proof}
    Let $\wb \in \Hb^1_0$ solve $-\vecDelta \wb = \ub - \ub_h$. Denote $\lambdab := -\curl\wb$. Since $\Omega$ is assumed convex, by Theorem~\ref{thm:reglap} we have 
    \begin{equation}\label{eq:thml2est-reg}
        \|\wb\|_{\Hb^2} +\|\lambdab\|_{\Hb^1} \leq C\|\ub - \ub_h\|. 
    \end{equation}
    By the consistency of~\eqref{eq:hodgelapdisc}, it holds that
    \begin{equation}\label{eq:thml2est-1}
        \begin{aligned}
            \|\ub-\ub_h\|^2 &= \a((\mub-\mub_h,\ub-\ub_h), (\lambdab, \wb)) \\
            &= \a((\mub-\mub_h,\ub-\ub_h),(\lambdab - \lambdab_\Pi, \wb - \wb_I)) 
        \end{aligned}
    \end{equation}
    where we set $\wb_I := \IRT\wb\in \Vb_{h,0}$ and $\lambdab_\Pi := \PND\lambdab \in \Sigmabh$ as defined in~\eqref{def:proj}. 
    In view of the definition of the bilinear form $\a$ (see~\eqref{def:a}), we have by~\eqref{eq:thml2est-1} 
    \begin{equation}\label{eq:thml2est-split}
        \begin{aligned}
            \|\ub-\ub_h\|^2 &= \underbrace{(\div (\wb-\wb_I),\div(\ub-\ub_h))}_{I_1} + \underbrace{(\curl (\mub - \mub_h), \wb - \wb_I)}_{I_2}\\
            &\quad\quad - \underbrace{(\curl (\lambdab - \lambdab_\Pi), \ub -\ub_h)}_{I_3} + \underbrace{(\lambdab - \lambdab_\Pi, \mub-\mub_h)}_{I_4}.
        \end{aligned}
    \end{equation}
    
    Using the canonical interpolation~\eqref{eq:canonicalapp}, we have $I_1 \leq h\|\div(\ub-\ub_h)\||\div\wb|_{H^1}$.

    When $k \geq 2$, we have by~\eqref{eq:canonicalapp} that $I_2 \leq  Ch\cdot h\|\curl(\mub - \mub_h)\||\wb|_{\Hb^2}.$
    Estimating $I_2$ in the lowest-order case $k=1$ is slightly more involved. Denote $\mub_\Pi := \PND\mub\in \Sigmabh$. Recall the projection $\Pzh$ defined in~\eqref{def:tau0h}. We proceed with
    \begin{align}
        I_2 = \underbrace{(\curl(\mub-\mub_\Pi),\wb-\wb_I)}_{I_{2,1}} + \underbrace{(\curl((\mub_\Pi-\mub_h) - \Pzh(\mub_\Pi-\mub_h)),\wb-\wb_I)}_{I_{2,2}} + \underbrace{(\curl\Pzh(\mub_\Pi-\mub_h),\wb-\wb_I)}_{I_{2,3}}.\label{eq:thml2est-I_2}
    \end{align}
    Due to~\eqref{eq:curlprojerrcurl} and~\eqref{eq:canonicalapp}, we have 
    $I_{2,1} \leq Ch\|\curl\mub\||\wb|_{\Hb^1}$.
    Applying~\eqref{eq:thmerr-4} in the case of $q = 6/5$, the interpolation approximation~\eqref{eq:canonicalapp} and the $\Sigmabh$-norm equivalence (see Lemma~\ref{thm:sigmanormequivdisc}), we obtain
    \begin{equation}
        \begin{aligned}
            I_{2,2} &\leq  \|\curl((\mub_\Pi-\mub_h) - \Pzh(\mub_\Pi-\mub_h))\|_{\Lb^{6/5}}\|\wb-\wb_I\|_{\Lb^6}\\
            &\leq Ch^{1/3}\|(\mub_\Pi-\mub_h) - \Pzh(\mub_\Pi-\mub_h)\||\wb|_{\Wb^{1,6}} \\
            &\leq Ch^{1/3}\|\mub_\Pi-\mub_h\|_{\Sigmabh}\|\wb\|_{\Hb^2} \\
            &\leq Ch^{1/3}(\|\mub - \mub_h\|_{\Sigmabh} + \|\mub - \mub_\Pi\|_{\Sigmabh})\|\wb\|_{\Hb^2} \\
            &\leq Ch^{1/3}(\|\mub-\mub_h\|_{\Sigmabh} + h^{1/2}\|\mub\|_{\Hb^1})\|\wb\|_{\Hb^2}
        \end{aligned}
    \end{equation}
    where we have used the Sobolev embedding $H^{2} \hookrightarrow W^{1,6}$ in the third step and applied~\eqref{eq:superconv}\eqref{eq:curlprojerrl2} in the last step (notice that $\mub \in \Xb$). Similarly, we have
    \begin{equation}
        I_{2,3} \leq Ch\|\mub_\Pi - \mub_h\|_{\Sigmabh}|\wb|_{\Hb^1} \leq Ch(\|\mub-\mub_h\|_{\Sigmabh} + h^{1/2}\|\mub\|_{\Hb^1})|\wb|_{\Hb^1}.
    \end{equation}
    Apparently, $I_{2,2}$ dominates $I_2$.
    
    To address $I_3$, we introduce $\ub_I := \IRT\ub$ and deduce
    \begin{equation}
        \begin{aligned}
            I_3 &= \underbrace{(\curl(\lambdab - \lambdab_\Pi), \ub - \ub_I)}_{I_{3,1}} + \underbrace{(\curl(\lambdab - \lambdab_\Pi), \ub_I - \ub_h)}_{I_{3,2}}.
        \end{aligned}
    \end{equation}
    Using~\eqref{eq:canonicalapp} and~\eqref{eq:curlprojerrcurl}, we have $I_{3,1} \leq Ch^k\|\curl\lambdab\||\ub|_{\Hb^{k}}$. 
    For $I_{3,2}$, we apply~\eqref{eq:curlprojerrcurl}\eqref{eq:superconv}\eqref{eq:canonicalapp} and obtain
    \begin{equation}
        \begin{aligned}
            I_{3,2} &\leq Ch^{1/2}\|\curl\lambdab\|\|\div(\ub_I-\ub_h)\| \\
            & \leq Ch^{1/2}\|\curl\lambdab\|\left(\|\div(\ub_I-\ub)\| + \|\div(\ub-\ub_h)\|\right)\\
            & \leq Ch^{1/2}\|\curl\lambdab\|\left(h^{k}|\ub|_{\Hb^{k+1}} +  \|\div(\ub-\ub_h)\|\right).
        \end{aligned}
    \end{equation}

    It remains to bound the term $I_4$ in~\eqref{eq:thml2est-split}. Notice that $\lambdab \in \Xb$, and thus~\eqref{eq:curlprojerrl2} is applicable. We have then $I_4 \leq Ch^{1/2}\|\lambdab\|_{\Hb^1}\|\mub - \mub_h\|$.

    Collecting all these estimates, applying~\eqref{eq:thml2est-reg} and inserting~\eqref{eq:eneerrest}, we obtain~\eqref{eq:l2errest}.
\end{proof}

\section{Application to the Stokes problem}\label{sec:stokes}
Using~\eqref{eq:vvpdisc} as a discretization of~\eqref{eq:veclap} does not seem attractive at all, as the standard Lagrange element works perfectly for~\eqref{eq:veclap}. However, discretizing the vector (Dirichlet) Laplacian in the form of~\eqref{eq:vvpdisc} is desirable when solving certain problems. In this section, we discuss a valid discretization of the Stokes problem that arises naturally from~\eqref{eq:vvpdisc}.

Let $\Omega \subset \R^3$ be a bounded Lipschitz polygonal domain with one connected component. Consider the Stokes problem
\begin{equation}\label{eq:stokes}
    \begin{aligned}
        - \vecDelta \ub +\grad p &= \fb &&\text{ in }\Omega, \\
        \div \ub &= g &&\text{ in }\Omega,\\
        \ub &= \zerob &&\text{ on }\partial\Omega,
    \end{aligned}
\end{equation}
where $\fb \in \Hzdiv'$ and $g \in \Ltwoz$.
A substantial number of finite element discretizations have been proposed for solving~\eqref{eq:stokes}, featuring different properties. Classical schemes include the Taylor-Hood element, the mini-element, the Scott-Vogelius element, etc., while more recent works focus on schemes with the exact divergence-free property under the guidance of the Stokes complex; see~\cite{john_2017,neilan_2020} and references therein.

One of the feasible discretizations of~\eqref{eq:stokes} is to base it on the mixed form discretization of the Hodge Laplace BVP~\eqref{eq:potvecdisc}, enforcing the divergence condition via a Lagrange multiplier representing the pressure $p$. Introducing the vorticity of the fluid $\mub := \curl\ub$ as a new unknown, one can reformulate~\eqref{eq:stokes} as seeking $(\mub, \ub, p) \in \Sigmab \times \Hzdiv \times \Ltwoz$ such that
\begin{equation}\label{eq:vvpcont}
    \begin{aligned}
        (\mub, \taub) - \langle\curl \taub, \ub\rangle \mathcolor{white}{- (p,\div\vb)}&= 0 &&\forall\,\taub \in \Sigmab, \\
        \langle\curl\mub, \vb\rangle + (\div\ub, \div\vb) - (p,\div\vb)&= \langle\fb, \vb\rangle && \forall\,\vb\in \Hzdiv, \\
        (q,\div \ub) \mathcolor{white}{- (p,\div\vb)} &= (g,q) &&\forall\,q\in \Ltwoz.
    \end{aligned}
\end{equation}
For the same reason as for~\eqref{eq:hodgelap}, the naive choice $\Sigmab = \Hcurl$ leads to an ill-posed problem, and the remedy is to consider $\Sigmab$ defined in~\eqref{def:Sigmab}.

Using standard FEEC-type elements (see~\eqref{def:spaces}), we seek $(\mub_h, \ub_h, p_h) \in \Sigmabh \times  \Vb_{h,0} \times S_{h,0}$ such that
\begin{equation}\label{eq:vvpdisc}
    \begin{aligned}
        (\mub_h, \taub_h) - (\curl \taub_h, \ub_h) \mathcolor{white}{- (p_h,\div\vb_h)}&= 0 &&\forall\,\taub_h \in \Sigmabh, \\
        (\curl\mub_h, \vb_h) + (\div\ub_h, \div\vb_h) - (p_h,\div\vb_h)&= \langle\fb, \vb_h\rangle && \forall\,\vb_h\in\Vb_{h,0}, \\
        (q_h,\div \ub_h) \mathcolor{white}{- (p_h,\div\vb_h)}&= (g, q_h) &&\forall\,q_h\in S_{h,0}.
    \end{aligned}
\end{equation}
This formulation is called the \emph{Vorticity-Velocity-Pressure} (VVP) formulation~\cite{dubois_2003,dubois_2003b,arnold_2012}. The use of an $\Hdiv$-conforming discrete velocity space ensures an exactly divergence-free velocity, resulting in desirable pressure-robust properties~\cite{john_2017}. Additionally, such a formulation is particularly popular in solving coupled Darcy-Stokes problems~\cite{bernardi_2005,vassilevski_2014,alvarez_2016} and is part of the structure-preserving dual-field discretization of the Navier-Stokes equations~\cite{zhang_2022b}.  

\begin{remark}[keeping $(\div\ub,\div\vb)$]
    By considering $p' := p - \div\ub$, the term $(\div\ub,\div\vb)$ in~\eqref{eq:vvpcont} can be discarded. However, we keep this term so that we can reuse the bilinear $\mathsf{a}$ defined as~\eqref{def:a}.
\end{remark}
\begin{remark}[other VVP-formulations]\label{rmk:vvp}
    Alternatively, one may attempt to seek the vorticity $\mub \in \Hdiv$, the velocity $\ub \in \Hzcurl$, and the pressure $p \in H^1/\R$ such that
    \begin{equation}\label{eq:vvpcurl}
        \begin{aligned}
            (\mub, \taub) - (\curl \ub, \taub) &= 0 &&\forall\,\taub \in \Hdiv, \\
            (\mub, \curl \vb) + (\nabla p, \vb) &= (\fb, \vb) &&\forall\,\vb \in \Hzcurl, \\
            (\nabla q,\ub) &= (-g, q) &&\forall\,q \in H^1 /\R.
        \end{aligned}
    \end{equation}
    According to different spaces for the velocity, we refer to~\eqref{eq:vvpcont} as \emph{``$\Hdiv$-based Dirichlet Stokes''} and~\eqref{eq:vvpcurl} as \emph{``$\Hcurl$-based Dirichlet Stokes''}. Putting aside the well-posedness\footnote{For a similar reason of using $\Sigmab \supset \Hcurl$ in the $\Hdiv$-based Dirichlet Stokes, $p$ has to be sought in a suitable superset of $H^1$ to ensure well-posedness.} of~\eqref{eq:vvpcurl}, failure of the FEEC discretizations of~\eqref{eq:vvpcurl} for certain meshes has been pointed out in~\cite[][Section~3.2]{wouter_2026}. Nonetheless, VVP formulations using various spaces are proposed to solve Stokes-type problems subject to various boundary conditions; see Table~\ref{tab:vvp} for a non-exhaustive list. 
    \begin{table}[!htbp]
        \centering
        \begin{tabular}{cccc}
             \hline\hline
             vorticity & velocity & pressure  & references \\
             \hline
             $\Hcurl$ & $\Hdiv$  & $L^2$ & \cite{dubois_2003, bernardi_2005, arnold_2012, vassilevski_2014, anaya_2016, alvarez_2016,zhang_2022b} \\
             $\Hdiv$  & $\Hcurl$ & $H^1$ & \cite{chami_2012,zhang_2022b} \\
             $\Hcurl$ & $\Lb^2$  & $H^1$ & \cite{bernardi_2010, wouter_2026} \\
             $\Hb^1$  & $\Lb^2$  & $L^2$ & \cite{anaya_2021,badia_2025}\\
             \hline\hline
        \end{tabular}
        \caption{A non-exhaustive list of works involving VVP formulations of the Stokes problems using different spaces for the vorticity, velocity and pressure. }
        \label{tab:vvp}
    \end{table}
\end{remark}

\subsection{A priori error analysis}
Comparing the Stokes problem~\eqref{eq:vvpcont} (and its discretization~\eqref{eq:vvpdisc}) with the Hodge Laplace BVP~\eqref{eq:hodgelap} (and its discretization~\eqref{eq:hodgelapdisc}), the former can be interpreted as a saddle-point problem based on the latter. Hence, the analysis of~\eqref{eq:vvpcont}\eqref{eq:vvpdisc} involves combining saddle-point theory (see, e.g.,~\cite[][Chapter~12]{brenner_2007}) with the previously established results for the Hodge Laplace BVP.

Recall the bilinear form $\a: \Sigmab \times \Hzdiv \rightarrow \R$ defined in~\eqref{def:a} and introduce the bilinear form $\b: \Ltwoz \times \Hzdiv \rightarrow \R$ by
\begin{equation}\label{def:b}
    \b(q,\vb) := (q, \div\vb).
\end{equation}
The systems~\eqref{eq:vvpcont} and~\eqref{eq:vvpdisc} can be rewritten as
\begin{equation}\label{eq:vvpcontab}
    \begin{aligned}
        \a((\mub,\ub),(\taub,\vb)) - \b(p,\vb) &= \langle \fb,\vb\rangle && \forall\, (\taub,\vb)\in \Sigmab\times\Hzdiv, \\
        \b(q,\ub) &= (g, q) &&\forall\,q\in \Ltwoz,
    \end{aligned}
\end{equation}
and 
\begin{equation}\label{eq:vvpdiscab}
    \begin{aligned}
        \a((\mub_h,\ub_h),(\taub_h,\vb_h)) - \b(p_h,\vb_h) &= \langle \fb,\vb_h\rangle && \forall\, (\taub_h,\vb_h)\in \Sigmabh\times\Vb_{h,0}, \\
        \b(q_h,\ub_h) &= (g, q_h) &&\forall\,q_h\in S_{h,0},
    \end{aligned}
\end{equation}
respectively. We introduce the following subspaces of $\Hzdiv$ and $\Vb_{h,0}$:
\begin{equation}\label{def:ZZh}
    \Zb := \{\vb \in \Hzdiv: \b(q, \vb) = 0\;\forall\,q\in \Ltwoz \},\quad\Zb_h := \{\vb_h \in \Vb_{h,0}: \b(q_h,\vb_h)=0\;\forall\,q_h\in S_{h,0}\},
\end{equation}
and the following affine spaces:
\begin{equation}\label{def:ZZhaff}
    \Zb^g := \{\vb \in \Hzdiv: \b(q, \vb) = (g,q)\;\forall\,q\in \Ltwoz \},\quad\Zb_h^g := \{\vb_h \in \Vb_{h,0}: \b(q_h,\vb_h)=(g,q_h)\,\forall\,q_h\in S_{h,0}\}.
\end{equation}
Since $\mathrm{div}\Hzdiv = \Ltwoz$ and $\mathrm{div}\Vb_{h,0} = S_{h,0}$, it holds that
\begin{equation}
    \Zb = \Hzdivz,\quad\ \Zb_h = \Vb_{h,0} \cap \Zb.
\end{equation}

First, the inf-sup conditions are established.
\begin{lemma}[inf-sup conditions]\label{thm:infsupvvpstokes}
    The following holds:

    \noindent $(i)$ there exist positive constants $\beta_\a,\beta_\b$ such that
    \begin{align}
        \sup_{0\neq (\taub, \vb)\in\Sigmab\times\Zb}\frac{\a((\mub,\ub),(\taub,\vb))}{\|(\taub,\vb)\|_{\Sigmab\times\Hdiv}} &\geq \beta_\a\|(\mub,\ub)\|_{\Sigmab\times\Hdiv} && \forall\,(\mub,\ub)\in \Sigmab\times\Zb,\label{eq:infsupastokescont}\\
        \sup_{\zerob\neq \vb \in \Hzdiv}\frac{\b(q,\vb)}{\|\vb\|_{\Hdiv}} &\geq \beta_\b \|q\| &&\forall\,q \in \Ltwoz;\label{eq:infsupbstokescont}
    \end{align}

    \noindent $(ii)$ there exist positive constants $\beta_\a',\beta_\b'$ independent of $h$ such that
    \begin{align}
        \sup_{0\neq (\taub_h, \vb_h)\in\Sigmabh\times\Zb_h}\frac{\a((\mub_h,\ub_h),(\taub_h,\vb_h))}{\|(\taub_h, \vb_h)\|_{\Sigmabh\times\Hdiv}} &\geq \beta_\a'\|(\mub_h, \ub_h)\|_{\Sigmabh\times\Hdiv} && \forall\,(\mub_h,\ub_h)\in \Sigmabh\times\Zb_h,\label{eq:infsupastokesdisc} \\
        \sup_{\zerob\neq \vb_h\in \Vb_{h,0}}\frac{\b(q_h,\vb_h)}{\|\vb_h\|_{\Hdiv}} &\geq \beta_\b' \|q_h\| &&\forall\,q_h \in S_{h,0}.\label{eq:infsupbstokesdisc}
    \end{align}
\end{lemma}
\begin{proof}
    The condition~\eqref{eq:infsupbstokescont} and~\eqref{eq:infsupbstokesdisc} result from $\mathrm{div}\Hzdiv = \Ltwoz,\,\mathrm{div}\Vb_{h,0} = S_{h,0}$ and the (discrete) \Poincare inequalities. The condition~\eqref{eq:infsupastokescont} and~\eqref{eq:infsupastokesdisc} differ from~\eqref{eq:infsupa} and~\eqref{eq:infsupadisc} with $\Hzdiv$ replaced by $\Zb$ and $\Vb_{h,0}$ replaced by $\Zb_h$, respectively. Yet the very same proof applies, noting that $\|\vb\|_{\Hdiv} = \|\vb\|$ for $\vb \in \Zb$.
\end{proof}
\noindent Obviously, the bilinear form $\b$ is bounded with respect to $L^2\times\Hdiv$. Recall the boundedness of the bilinear form $\a$ (see Lemma~\ref{thm:bddinfsup} and~\ref{thm:bddinfsupdisc}). Together with Lemma~\ref{thm:infsupvvpstokes}, the Babu\v ska-Brezzi theorem for saddle-point problems (see, e.g.,~\cite[][Theorem~49.13]{ern_2021b}) implies the well-posedness of~\eqref{eq:vvpcont}\eqref{eq:vvpdisc}.
\begin{theorem}[well-posedness of the VVP Stokes]
    The problem~\eqref{eq:vvpcont} and~\eqref{eq:vvpdisc} admit unique solutions $(\mub,\ub,p) \in \Sigmab \times \Hzdiv \times \Ltwoz$ and $(\mub_h, \ub_h, p_h) \in \Sigmabh \times  \Vb_{h,0} \times S_{h,0}$, respectively. There exists a positive constant $C$ independent of $\fb \in \Hzdiv'$ and $g \in \Ltwoz$ such that
    \begin{equation}
        \|(\mub,\ub,p)\|_{\Sigmab \times \Hdiv \times \Ltwo}  + \|(\mub_h, \ub_h, p_h)\|_{\Sigmabh \times \Hdiv \times \Ltwo} \leq C(\|\fb\|_{\Hzdiv'} + \|g\|).
    \end{equation}
\end{theorem}

Next, we derive the \emph{a priori} error estimates.
\begin{theorem}[error estimates of $\ub$ and $p$]\label{thm:hdiverreststokes}
    Let $(\mub,\ub, p)\in \Sigmab\times\Hzdiv\times\Ltwoz$ solve~\eqref{eq:stokes} and $(\mub_h, \ub_h, p_h) \in \Sigmabh \times  \Vb_{h,0} \times S_{h,0}$ solve~\eqref{eq:vvpdisc}. If $\ub \in \Wb^{k,\infty}\cap \Hb^{k+1}$ and $p \in \Hb^k$, there exists a positive constant $C$ such that
    \begin{align}
        \|\ub - \ub_h\| &\leq C h^{k - 1/2}\left(|\ub|_{\Wb^{k,\infty}} + \|\ub\|_{\Hb^{k+1}}\right), \label{eq:uerreststokes}\\
        \|p - p_h\| &\leq C h^{k - 1/2}\left(h^{1/2}|p|_{H^k} + |\ub|_{\Wb^{k,\infty}} + \|\ub\|_{\Hb^{k+1}}\right), \label{eq:perreststokes} \\
        \|\div(\ub-\ub_h)\|  &\leq C h^k|\div \ub|_{H^k}.\label{eq:diverreststokes}
    \end{align}
     Assuming $\Omega$ is convex in addition, it holds that
     \begin{equation}\label{eq:l2erreststokes}
        \|\ub - \ub_h\| \leq \begin{cases}
            C h^{k}\left(|\ub|_{\Wb^{k,\infty}} + \|\ub\|_{\Hb^{k+1}}\right), & k \geq 2,  \\
            C h^{5/6}\left(|\ub|_{\Wb^{1,\infty}} + \|\ub\|_{\Hb^{2}}\right), & k = 1. 
        \end{cases}.
    \end{equation}
\end{theorem}
\begin{proof}
    We first reduce the full problem to the Hodge Laplacian by eliminating the pressure.
    Recast~\eqref{eq:vvpcontab} and~\eqref{eq:vvpdisc} into
    \begin{align}
        &\text{seeking }(\mub,\ub) \in \Sigmab\times\Zb^g :\quad\quad\quad\a((\mub,\ub),(\taub,\vb)) = \langle\fb,\vb\rangle \quad &&\forall (\taub, \vb)\in\Sigmab\times\Zb,\label{eq:vvpcontrd}\\
        \text{and }&\text{seeking }(\mub_h,\ub_h) \in \Sigmabh\times\Zb_h^g :\quad\a((\mub_h,\ub_h),(\taub_h,\vb_h)) = \langle\fb,\vb_h\rangle \quad &&\forall (\taub_h, \vb_h)\in\Sigmabh\times\Zb_h.\label{eq:vvpdiscrd}
    \end{align}
    Note the conformity $\Sigmabh \subset \Sigmab, \Zb_h \subset \Zb$.
    Denote $\ub_I := \IRT \ub \in \Vb_{h,0}$ and $\mub_\Pi := \PND \mub \in \Sigmabh$ (recall the definitions of $\IRT$ and $\PND$ in Section~\ref{sec:interapprox}). Due to the commutative property $\div\IRT \ub = \PLG \div \ub$ where $\PLG:\Ltwoz\rightarrow S_{h,0}$ is the $L^2$-projection, we have $\ub_I \in \Zb^g_h$, and thus $\ub_I - \ub_h \in \Zb_h$.
    By the inf-sup condition~\eqref{eq:infsupastokesdisc} and~\eqref{eq:vvpcontrd}\eqref{eq:vvpdiscrd}, we deduce that
    \begin{equation}\label{eq:thmerreststokes-5}
        \begin{aligned}
            \|\ub_I - \ub_h\|_{\Hdiv} + \|\mub_\Pi - \mub_h\|_{\Sigmabh} &\leq \beta^{'-1}_\a \sup_{0\neq (\taub_h, \vb_h)\in\Sigmabh\times\Zb_h} \frac{\a((\mub_\Pi - \mub_h,\ub_I - \ub_h),(\taub_h,\vb_h))}{\|\taub_h\|_{\Sigmabh} + \|\vb_h\|_{\Hdiv}} \\
            &=\beta^{'-1}_\a \sup_{0\neq (\taub_h, \vb_h)\in\Sigmabh\times\Zb_h}    \frac{\a((\mub_\Pi - \mub,\ub_I - \ub),(\taub_h,\vb_h))}{\|\taub_h\|_{\Sigmabh} + \|\vb_h\|_{\Hdiv}}.
        \end{aligned}
    \end{equation}
    The same argument in the proof of Theorem~\ref{thm:eneerrest} applies seamlessly, and we obtain
    \begin{equation}\label{eq:eneerreststokes}
        \|\ub - \ub_h\|_{\Hdiv} + \|\mub - \mub_h\|_{\Sigmabh} + h\|\curl(\mub - \mub_h)\| \leq C h^{k - 1/2}\left(|\ub|_{\Wb^{k,\infty}} + \|\ub\|_{\Hb^{k+1}}\right)
    \end{equation}
    which settles the term $\|\ub - \ub_h\|$ in~\eqref{eq:uerreststokes}.
    Since $(\div \ub_h, q_h) = (\div \ub, q_h)\;\forall\,q_h\in S_{h,0}$ and $\div \Vb_{h,0} = S_{h,0}$, it is easily seen that $\div \ub_h$ is the $L^2$-projection of $\div \ub$ onto $S_{h,0}$, which leads to~\eqref{eq:diverreststokes}.

    To derive the error bound for $p$, we deduce by subtracting~\eqref{eq:vvpdiscab} from~\eqref{eq:vvpcontab} that
    \begin{equation}\label{eq:thmerreststokes-4}
        \b(p - p_h, \vb_h) = \a((\mub - \mub_h, \ub - \ub_h), (\taub_h, \vb_h)) \quad \forall\, (\taub_h,\vb_h)\in \Sigmabh\times\Vb_{h,0}.
    \end{equation}
    Using~\eqref{eq:infsupbstokesdisc}\eqref{eq:thmerreststokes-4} and letting $p_I := \PLG p$, we have
    \begin{equation}
        \begin{aligned}
            \|p_I - p_h\| &\leq \beta'^{-1}_\b \sup_{\zerob\neq \vb_h\in \Vb_{h,0}}\frac{\b(p_I -p_h, \vb_h)}{\|\vb_h\|_{\Hdiv}} \\
            &= \beta'^{-1}_\b \sup_{\zerob\neq \vb_h\in \Vb_{h,0}}\left(\frac{\b(p_I -p, \vb_h)}{\|\vb_h\|_{\Hdiv}} + \frac{\b(p -p_h, \vb_h)}{\|\vb_h\|_{\Hdiv}}\right) \\
            &= \beta'^{-1}_\b \sup_{\zerob\neq \vb_h\in \Vb_{h,0}}\left(\frac{\b(p_I -p, \vb_h)}{\|\vb_h\|_{\Hdiv}} + \frac{\a((\mub - \mub_h, \ub - \ub_h), (\zerob, \vb_h))}{\|\vb_h\|_{\Hdiv}}\right) \\
            &= \beta'^{-1}_\b \sup_{\zerob\neq \vb_h\in \Vb_{h,0}}\left(\frac{\b(p_I -p, \vb_h)}{\|\vb_h\|_{\Hdiv}} + \frac{(\div (\ub - \ub_h),\div \vb_h) + \langle\curl (\mub- \mub_h),\vb_h\rangle}{\|\vb_h\|_{\Hdiv}}\right) \\
            &\leq C \left(\|p_I - p\| + \|\div(\ub - \ub_h)\| + \|\mub - \mub_h\|_{\Sigmabh} \right).
        \end{aligned}
    \end{equation}
    Combining with~\eqref{eq:diverreststokes}\eqref{eq:thmerreststokes-5} leads to the error bound for $p$ in~\eqref{eq:perreststokes}.

    If $\Omega$ is convex, the error $\|\ub - \ub_h\|$ can be improved using the duality technique, as is done in the proof of Theorem~\ref{thm:l2errest}. Construct the dual problem 
    \begin{equation}\label{eq:stokesdual}
        \begin{aligned}
        - \vecDelta \wb +\grad \eta &= \ub-\ub_h &&\text{ in }\Omega, \\
        \div \wb &= 0 &&\text{ in }\Omega,\\
        \wb &= \zerob &&\text{ on }\partial\Omega.
    \end{aligned}
    \end{equation}
    Denote $\lambdab:=-\curl\wb$. Since $\Omega$ is assumed to be convex, we have $\wb \in \Hb^2, \eta \in H^1$~\cite{dauge_1989} and
    \begin{equation}\label{eq:thmerreststokes-3}
        \|\wb\|_{\Hb^2} + \|\lambdab\|_{\Hb^1} + \|\eta\|_{H^1} \leq C\|\ub-\ub_h\|.
    \end{equation}
    The dual problem~\eqref{eq:stokesdual} can be rewritten as
    \begin{equation}\label{eq:stokesdualrd}
        \a((\taub,\vb), (\lambdab, \wb)) = (\ub - \ub_h, \vb) \quad\forall\, (\taub,\vb)\in \Sigmab\times\Zb.
    \end{equation}
    Let $\varphi \in H^1/\R$ solve $-\Delta \varphi = \div (\ub - \ub_h) \in \Ltwoz$. By Theorem~\ref{thm:reglap} and~\eqref{eq:diverreststokes} we have 
    \begin{equation}\label{eq:thmerreststokes-2}
        \|\varphi\|_{H^2} \leq Ch^k|\div \ub|_{H^k}. 
    \end{equation}
    Note that $(\ub - \ub_h + \grad\varphi) \in \Zb$.
    Inserting $\taub = \mub - \mub_h, \vb = \ub - \ub_h + \grad\varphi$ into~\eqref{eq:stokesdualrd} and using~\eqref{eq:vvpcontrd}\eqref{eq:vvpdiscrd} yields
    \begin{equation}\label{eq:thmerreststokes-1}
        \begin{aligned}
            (\ub - \ub_h, \ub - \ub_h + \grad\varphi) &= \a((\mub-\mub_h, \ub - \ub_h + \grad\varphi), (\lambdab, \wb)) \\
            &= \a((\mub-\mub_h, \ub - \ub_h), (\lambdab, \wb)) + \a((\zerob, \grad\varphi), (\lambdab, \wb)) \\
            &= \underbrace{\a((\mub-\mub_h, \ub - \ub_h),(\lambdab - \lambdab_\Pi, \wb - \wb_I))}_{I_1} + \underbrace{\a((\zerob, \grad\varphi), (\lambdab, \wb))}_{I_2}
        \end{aligned}
    \end{equation}
    where $\lambdab_\Pi := \PND \lambdab \in \Sigmabh$ and $\wb_I := \IRT \wb \in \Zb_h$ since $\div\wb = 0$. The left-hand side of~\eqref{eq:thmerreststokes-1} satisfies that
    \begin{equation}\label{eq:thmerreststokes-8}
        (\ub - \ub_h, \ub - \ub_h + \grad\varphi) \geq \frac{1}{2}\|\ub - \ub_h\|^2 - \frac{1}{2}\|\grad \varphi\|^2.
    \end{equation}
    The term $I_1$ is estimated in the same way as in the proof of Theorem~\ref{thm:l2errest} (cf.~\eqref{eq:thml2est-1}). That is, 
    \begin{equation}\label{eq:thmerreststokes-6}
        I_1 \leq \begin{cases}
            C h^{k}\left(|\ub|_{\Wb^{k,\infty}} + \|\ub\|_{\Hb^{k+1}}\right)\|\ub - \ub_h\|, & k \geq 2,  \\
            C h^{5/6}\left(|\ub|_{\Wb^{1,\infty}} + \|\ub\|_{\Hb^{2}}\right)\|\ub - \ub_h\|, & k = 1. 
    \end{cases}.
    \end{equation}
    By~\eqref{eq:bddacont}\eqref{eq:thmerreststokes-3}\eqref{eq:thmerreststokes-2}, we bound $I_2$ as
    \begin{equation}\label{eq:thmerreststokes-7}
        I_2 \leq C(\|\lambdab\|_{\Sigmab} + \|\wb\|_{\Hdiv})\|\grad \varphi\|_{\Hdiv} \leq C\|\ub-\ub_h\| h^k|\div\ub|_{H^k}.
    \end{equation}
    Combine the estimates~\eqref{eq:stokesdualrd}\eqref{eq:thmerreststokes-8}\eqref{eq:thmerreststokes-6}\eqref{eq:thmerreststokes-7}, and we can conclude.
\end{proof}
\begin{remark}[optimal rate in $\Hdiv$-error]
    Comparing Theorem~\ref{thm:l2errest} and Theorem~\ref{thm:hdiverreststokes} when $\Omega$ is convex, we remark that the convergence rate of the $\Hdiv$-error is (almost) optimal for the Stokes problem, whereas it is $h^{1/2}$-suboptimal in $\Hdiv$-seminorm for the Hodge Laplace BVP. 
\end{remark}

\section{Numerical experiments}\label{sec:num}
In this section, we conduct numerical tests\footnote{The implementation is conducted using \href{https://ngsolve.org/index.html}{\texttt{NGSolve}}.} to verify the established convergence rates. We mention that some numerical results have also been presented in~\cite{dubois_2003b, arnold_2012}.
\subsubsection*{Case I: the Laplace BVP on a non-convex domain with void}
We consider $\Omega = (0,1)^3 \setminus [1/3,2/3]^3$ and the manufactured solution 
\begin{equation}
    \ub_{\text{exact}}(x,y,z) = 
        \begin{bmatrix}
        \sin(3\pi x)\sin(6\pi y)\sin(3\pi z) \\
        \sin(6\pi x)\sin(3\pi y)\sin(6\pi z) \\
        xyz(1/3-x)(1/3-y)(1/3-z)(2/3-x)(2/3-y)(2/3-z)(1-x)(1-y)(1-z)
        \end{bmatrix}.
\end{equation}
The convergence plot of the scheme~\eqref{eq:hodgelapdisc} is shown in Figure~\ref{fig:conv1}.
\begin{figure}
    \centering
    \includegraphics[scale=0.4, trim=0pt 0pt 430pt 0pt, clip]{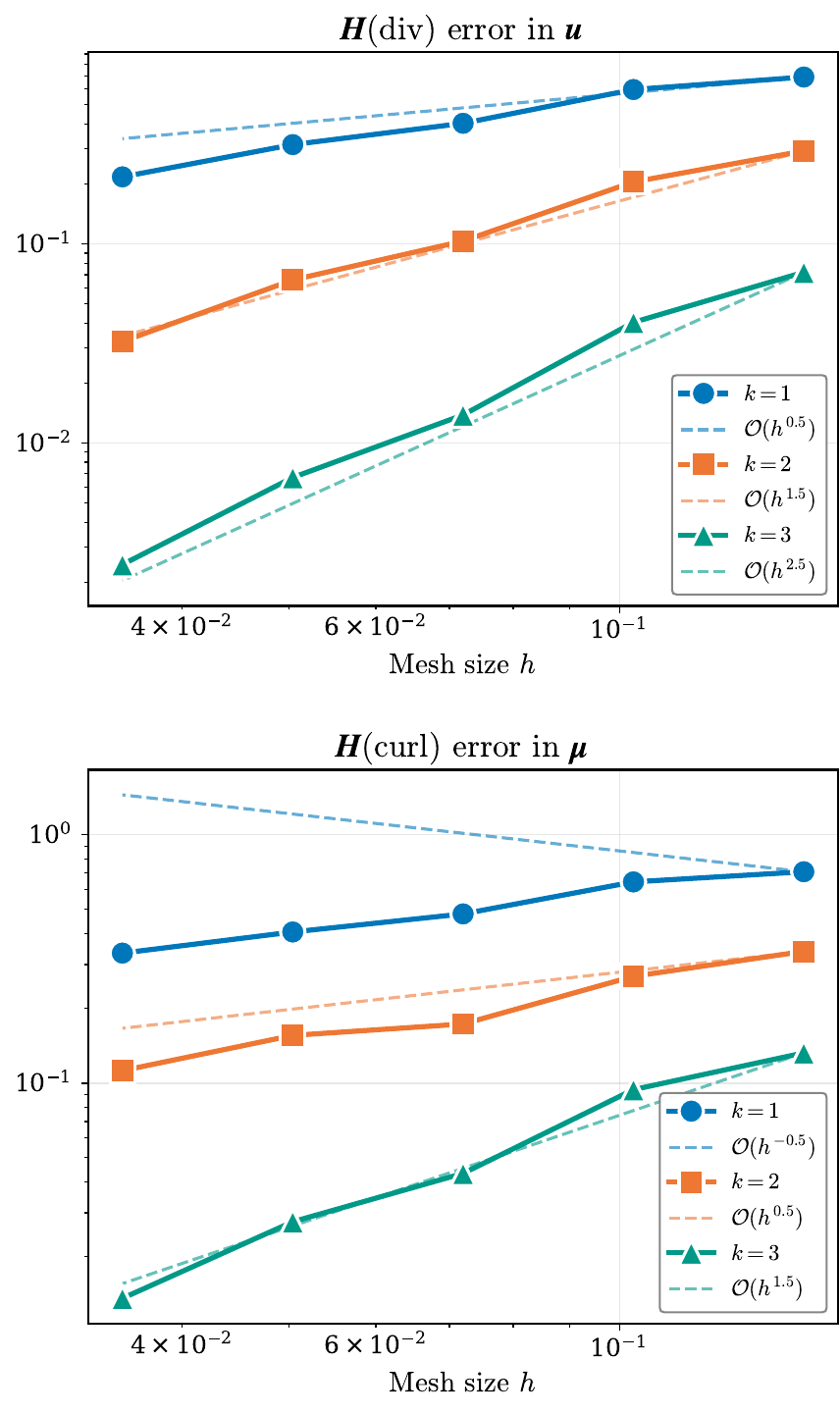}
    \hspace{1cm}
    \includegraphics[scale=0.4, trim=430pt 0pt 0pt 0pt, clip]{convergence_plot_nonconvex.pdf}
    \caption{Case I: Convergence study of the scheme~\eqref{eq:hodgelapdisc} on a domain with void}
    \label{fig:conv1}
\end{figure}

\subsubsection*{Case II: the Laplace BVP on a convex domain}
We consider $\Omega = (0,1)^3$ and the manufactured solution 
\begin{equation}
    \ub_{\text{exact}}(x,y,z) = 
        \begin{bmatrix}
        \sin(2\pi x)\sin(2\pi y)\sin(2\pi z) \\
        \sin(\pi x)\sin(\pi y)\sin(\pi z) \\
        x(1-x)y(1-y)z(1-z)
        \end{bmatrix}.
\end{equation}
The convergence plot of the scheme~\eqref{eq:hodgelapdisc} is shown in Figure~\ref{fig:conv2}.
\begin{figure}
    \centering
    \includegraphics[scale=0.4, trim=0pt 0pt 430pt 0pt, clip]{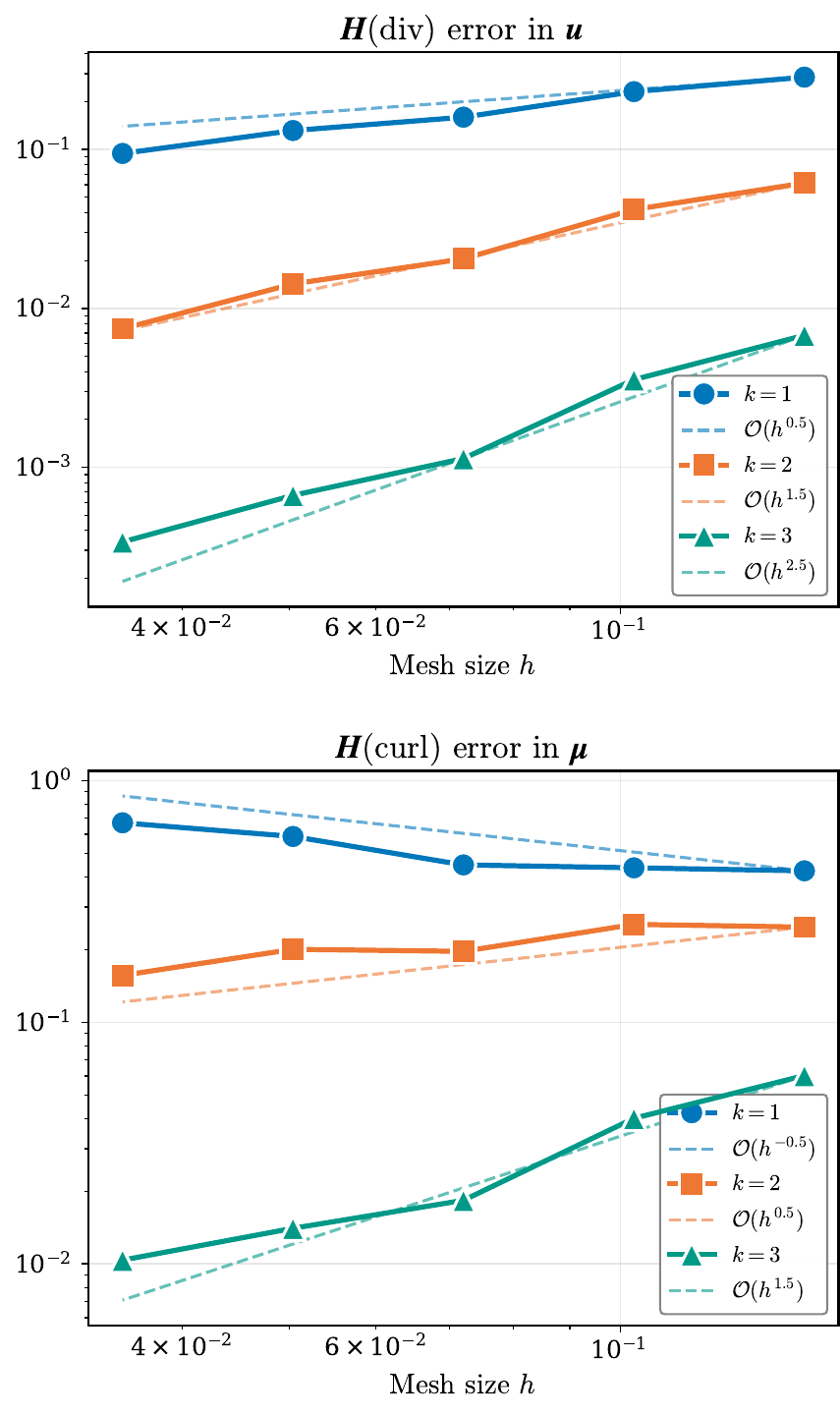}
    \hspace{1cm}
    \includegraphics[scale=0.4, trim=430pt 0pt 0pt 0pt, clip]{convergence_plot_convex.pdf}
    \caption{Case II: Convergence study of the scheme~\eqref{eq:hodgelapdisc} on a convex domain}
    \label{fig:conv2}
\end{figure}

\subsubsection*{Case III: the Stokes problem on a convex domain}
We consider $\Omega = (0,1)^3$ and the manufactured solution 
\begin{equation}
    \ub_{\text{exact}}(x,y,z) = 
        \begin{bmatrix}
        \sin(2\pi x)\sin(2\pi y)\sin(2\pi z) \\
        \sin(\pi x)\sin(\pi y)\sin(\pi z) \\
        x(1-x)y(1-y)z(1-z)
        \end{bmatrix},\quad\quad
    p_\text{exact}(x,y,z) = x^2\sin(2\pi y)\cos(4\pi z).
\end{equation}
The convergence plot of the scheme~\eqref{eq:vvpdisc} is shown in Figure~\ref{fig:conv3}.
\begin{figure}
    \centering
    \includegraphics[scale=0.4]{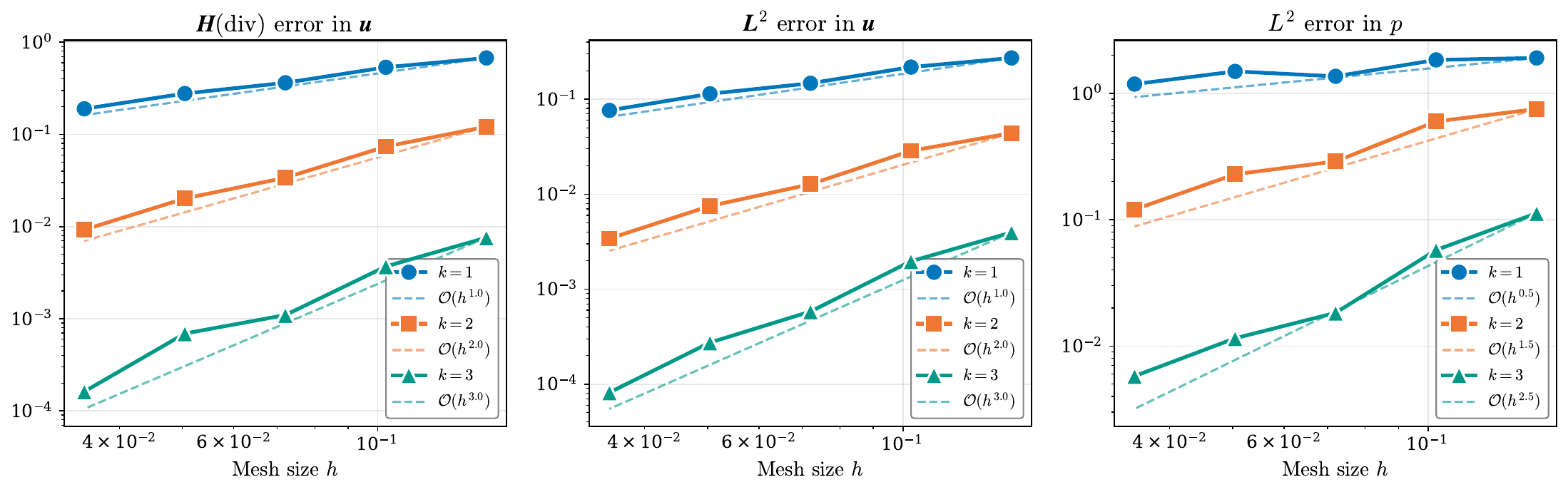}
    \caption{Case III: Convergence study of the scheme~\eqref{eq:vvpdisc} on a convex domain}
    \label{fig:conv3}
\end{figure}

\section{Concluding remarks}
We have analyzed the mixed finite element discretization~\eqref{eq:hodgelapdisc} of the vector Dirichlet Laplace boundary value problem~\eqref{eq:veclap} in three dimensions. Suitable inf-sup stability is established, and a priori error bounds are derived. Our results extend the two-dimensional analysis in~\cite{arnold_2012} to three-dimensional domains with nontrivial topology, requiring several analytical techniques different from those in~\cite{arnold_2012}. As a direct application, we also analyze the discretization of the Stokes equations in the VVP formulation~\eqref{eq:vvpdisc}. Numerical results agree well with the predicted convergence rates.

We would like to point out some potential gaps in the current theory. First, Figure~\ref{fig:conv1} indicates that the lowest-order elements might enjoy superconvergence; specifically, the $\Hdiv$-error of $\ub_h$ converges empirically at a rate of $O(h)$ even on a non-convex domain. Second, the $O(h^{5/6})$ convergence rate for the $\Lb^2$-error of $\ub_h$ (see Theorem~\ref{thm:l2errest}) on convex domains may not be sharp and could be an artifact of the analysis. We conjecture that a more refined technique could lift it to $O(h)$.

Regarding future research directions, the $\Hcurl$-version of the mixed approach can be treated similarly, although its application in the Stokes problem is currently less straightforward (see Remark~\ref{rmk:vvp}). Other aspects include spectral approximation ($p$-version), efficient preconditioning strategies, and applications in singularly-perturbed Darcy-Stokes problems.

\bibliography{ref_hdiv.bib}
\end{document}

%% file: fig_partition.tex

\begin{tikzpicture}[
    thick,
    line join=miter 
]
    \coordinate (C) at (0, 0);

    \coordinate (P1) at (0, 3);
    \coordinate (P2) at (3.5, 1.5);
    \coordinate (P3) at (2.5, -2.5);
    \coordinate (P4) at (-2, -3);
    \coordinate (P5) at (-4, 0);

    \def\drawlayer#1#2#3#4{
        \draw[fill=#2, draw=#3, #4]
            ($(C)!#1!(P1)$) -- 
            ($(C)!#1!(P2)$) -- 
            ($(C)!#1!(P3)$) -- 
            ($(C)!#1!(P4)$) -- 
            ($(C)!#1!(P5)$) -- cycle;
    }


    \drawlayer{1.0}{gray!20}{gray!60!black}{ultra thick}

    \drawlayer{0.90}{green!30!white}{green!40!black}{thick}

    \drawlayer{0.70}{cyan!30}{cyan!60!black}{thick}

    \drawlayer{0.30}{white}{cyan!60!black}{thick}

    
    \coordinate (ZoomTarget) at ($(P1)!0.4!(P2)$);
    
    \coordinate (ZoomCenter) at ($(C)!0.95!(ZoomTarget)$);
    
    \coordinate (MagCenter) at (7, 1);
    
    \draw[thick, red!80!white] (ZoomCenter) [xshift=-0.5cm, yshift=0.25cm] circle (0.25);
    
    \path (ZoomCenter) ++(-10:0.25) coordinate[xshift=-0.5cm, yshift=0.25cm] (ArrowStart);
    \path (MagCenter) ++(150:2.5) coordinate (ArrowEnd);
    \draw[->, thick, red!80!white, shorten >=2pt] (ArrowStart) to[out=-10, in=150] (ArrowEnd);

    \coordinate (labelCenter) at ($(P5)!0.5!(P1)$);
    \node at (-1,2.7) [rotate=38.2] {$\partial\Omega$};
    \node at ($(C)!0.95!(labelCenter)$)[text=gray!50!black,rotate=38.2] {\small$A_0$};
    \node at ($(C)!0.8!(labelCenter)$) [text=green!30!black,rotate=38.2]{\small$A_1$};
    \node at ($(C)!0.5!(labelCenter)$) [text=cyan!30!black,rotate=38.2]{\small$A_2$};
    \node at ($(C)!0.2!(labelCenter)$) {\small$\boldsymbol{\ddots}$};
    \draw[<->, thick, red!80!white] ($(C)!0.9!(ZoomTarget)$) -- ($(C)!0.7!(ZoomTarget)$) node[midway, right, inner sep=2pt, scale=0.8,rotate=-23.2] {${d_1 = 2d_0}$};
    \draw[<->, thick, red!80!white] ($(C)!0.7!(ZoomTarget)$) -- ($(C)!0.3!(ZoomTarget)$) node[midway, right, inner sep=2pt, scale=0.8,rotate=-23.2] {${d_2 = 2d_1}$};
    
    \begin{scope}[shift={(7, 1)}]
        
        \clip (0,0) circle (2.5);
        
        \begin{scope}[rotate=-23.2]
            
            \fill[gray!20] (-4, -1.2) rectangle (4, 2.0);          
            \fill[green!30!white] (-4, -4) rectangle (4, -1.2);   
            
            \draw[thick, green!40!black] (-4, -1.2) -- (4, -1.2);
            
            \draw[ultra thick, gray!60!black] (-4, 2.0) -- (4, 2.0);
            
            \begin{scope}
                \clip (-4, -4) rectangle (4, 2.0);
                
                \pgfmathsetseed{42} 
                
                \foreach \i in {-8,...,8} {
                    \foreach \j in {0,...,10} {
                        \pgfmathsetmacro{\bx}{\i * 0.6}
                        \pgfmathsetmacro{\by}{2.0 - \j * 0.6}
                        
                        \ifnum\j=0
                            \pgfmathsetmacro{\px}{\bx + 0.15*rand}
                            \pgfmathsetmacro{\py}{\by}
                        \else
                            \pgfmathsetmacro{\px}{\bx + 0.15*rand}
                            \pgfmathsetmacro{\py}{\by + 0.15*rand}
                        \fi
                        \coordinate (N-\i-\j) at (\px, \py);
                        
                        \expandafter\xdef\csname Nx-\i-\j\endcsname{\px}
                        \expandafter\xdef\csname Ny-\i-\j\endcsname{\py}
                    }
                }
                
                \foreach \i [evaluate=\i as \nexti using int(\i+1)] in {-8,...,7} {
                    \foreach \j [evaluate=\j as \nextj using int(\j+1)] in {0,...,9} {
                        \pgfmathsetmacro{\xA}{\csname Nx-\i-\j\endcsname}
                        \pgfmathsetmacro{\yA}{\csname Ny-\i-\j\endcsname}
                        \pgfmathsetmacro{\xB}{\csname Nx-\nexti-\j\endcsname}
                        \pgfmathsetmacro{\yB}{\csname Ny-\nexti-\j\endcsname}
                        \pgfmathsetmacro{\xC}{\csname Nx-\nexti-\nextj\endcsname}
                        \pgfmathsetmacro{\yC}{\csname Ny-\nexti-\nextj\endcsname}
                        \pgfmathsetmacro{\xD}{\csname Nx-\i-\nextj\endcsname}
                        \pgfmathsetmacro{\yD}{\csname Ny-\i-\nextj\endcsname}
                        
                        \pgfmathsetmacro{\dOne}{(\xC - \xA)*(\xC - \xA) + (\yC - \yA)*(\yC - \yA)}
                        \pgfmathsetmacro{\dTwo}{(\xD - \xB)*(\xD - \xB) + (\yD - \yB)*(\yD - \yB)}
                        
                        \ifdim\dOne pt < \dTwo pt
                            \draw[very thin, black!50] (N-\i-\j) -- (N-\nexti-\j) -- (N-\nexti-\nextj) -- cycle;
                            \draw[very thin, black!50] (N-\i-\j) -- (N-\i-\nextj) -- (N-\nexti-\nextj) -- cycle;
                        \else
                            \draw[very thin, black!50] (N-\i-\j) -- (N-\nexti-\j) -- (N-\i-\nextj) -- cycle;
                            \draw[very thin, black!50] (N-\nexti-\nextj) -- (N-\nexti-\j) -- (N-\i-\nextj) -- cycle;
                        \fi
                    }
                }
            \end{scope}
            \draw[<->, ultra thick, red!80!white] (1.2, -1.2) -- (1.2, 2.0) node[midway, fill=yellow!20, inner sep=2pt, rotate=-23.2] {$d_0 = ch$};
            
        \end{scope}
    \end{scope}
    
    \draw[thick, red!80!white] (MagCenter) circle (2.5);

\end{tikzpicture}